\documentclass[finalsize,finaltitle]{zaa}
\newcommand{\mycomment}{}

\usepackage{amssymb,amsmath,amsthm}
\usepackage{graphicx}
\usepackage{float}
\restylefloat{figure}

\newcommand{\beq}[1]{ \begin{equation}\label{#1} }
\newcommand{\eeq}{\end{equation}}
\numberwithin{equation}{section}
\DeclareMathOperator*{\esssup}{ess\,sup}
\DeclareMathOperator*{\essinf}{ess\,inf}

\def\RR{{\mathbb R}}

\begin{document}  

\newtheorem{guess}{Theorem}
\newtheorem{uess}{Lemma}

\title{On Stability of Delay Equations with\\
Positive and Negative Coefficients \\
with Applications}

\runtitle{Stability of Equations with Variable Delays}

\author{Leonid Berezansky and Elena Braverman}
\runauthor{L.~Berezansky and E.~Braverman}

\address{L.~Berezansky: Department of Mathematics,
Ben-Gurion University of the Negev, 
Beer-Sheva 84105, Israel}
\address{E.~Braverman: Department of Mathematics and Statistics, 
University of Calgary, 2500 University Dr. NW, Calgary, AB T2N 1N4, Canada.
\email{maelena@ucalgary.ca}}

\abstract{We obtain new explicit exponential stability conditions
for linear scalar equations with positive and negative delayed terms
$$
\dot{x}(t)+ \sum_{k=1}^m a_k(t)x(h_k(t))- \sum_{k=1}^l b_k(t)x(g_k(t))=0
$$
and its modifications, 
and apply them to investigate  local stability of Mackey--Glass type models
$$\dot{x}(t)=r(t)\left[\beta\frac{x(g(t))}{1+x^n(g(t))}-\gamma x(h(t))\right]$$
and
$$\dot{x}(t)=r(t)\left[\beta\frac{x(g(t))}{1+x^n(h(t))}-\gamma x(t)\right].$$
}
%\end{abstract}

%\begin{document}
%\begin{keyword}
%{\bf Keywords:} 
%integro-differential equations; equations with a distributed delay,  Mackey--Glass Equation.

%\noindent
%{\bf AMS Subject Classification:} 
%{\bf 2010 AMS subject classification:} 
\primclass{34K20} 
\secclasses{34K06, 92D25}

%\end{keyword}

\received{November 24, 2017}
\revised{August 31, 2018}
%\logo{??}{2018}{?}{?}
\logo{XX}{20XX}{X}{1}
\doi{}

\maketitle

\section{Introduction}

The Mackey--Glass equation with two delays and non-monotone feedback
\begin{equation}\label{51}
\dot{x}(t)=r(t)\left[\beta\frac{x(g(t))}{1+x^n(g(t))}-\gamma x(h(t))\right]
\end{equation}
generalizes the classical blood production model \cite{Mak1}
\begin{equation}\label{intro1}
\dot{x}(t)= \beta\frac{x(t-\tau)}{1+x^n(t-\tau)}-\gamma x(t).
\end{equation}
Similarly to (\ref{intro1}), if $\gamma<\beta$
then equation (\ref{51}) has a positive equilibrium 
$x^{\ast}=\big(\frac{\beta}{\gamma}-1\big)^{\frac{1}{n}}$.
The linearization of (\ref{51}) around $x^{\ast}$ has the form 
%of the equation with a positive and a negative coefficients

\begin{equation}\label{52}
\dot{x}(t)+r(t)\left[ \gamma  x(h(t))-\gamma \left( 1-n+\frac{\gamma n}{\beta} \right)x(g(t)) \right] =0.
\end{equation}

Another equation  generalizing (\ref{intro1}), with two delays involved in the production function
\begin{equation}
\label{MG}
\dot{x}(t) = s(t) \left[ \frac{\beta x(p(t))}{1+x^{n}(q(t))} - x(t) \right], \quad t \geq t_0,
\end{equation}
was recently considered in \cite{JMAA2017}. It has a positive equilibrium $x^{\ast}=(\beta-1)^{\frac{1}{n}}$ 
for $\beta>1$, and its linearization about 
$x^{\ast}$ has the form 
\begin{equation}\label{lin}
\dot{x}(t)+s(t)[x(t)+\alpha x(q(t))-x(p(t))]=0,
\end{equation}
where
\begin{equation*}\label{alpha}
\alpha=\frac{n(\beta-1)}{\beta}.
\end{equation*}

This motivates us to investigate stability of linear equations with several positive and negative terms.

To this end, we obtain new explicit exponential stability conditions
for the scalar delay differential equation with positive and negative coefficients
\begin{equation}\label{1}
\dot{x}(t) \!+ \!a(t)x(h(t)) \!- \!b(t)x(g(t)) \!= \!0, \quad a(t) \!\geq \! b(t) \!\geq \! 0, \quad h(t) \!\leq \! t, \quad g(t) \!\leq \! t,
\end{equation}
%$$
%a(t)\geq b(t)\geq 0, h(t)\leq t, g(t)\leq t,
%$$
and for some generalizations of this equation, including equations with several delays,
integro-differential equations and equations with distributed delays.

Most stability results for linear delay differential equations were obtained for equations
with positive coefficients, see, for example, \cite{Kz,SYC,GH1,GH2}. 
There are only few results for equations of type (\ref{1}).
The paper \cite{BB1} involves a review of stability tests for equation (\ref{1}). Most of these results
are obtained for the equation
\begin{equation}\label{2}
\dot{x}(t)+a(t)x(h(t))-b(t)x(t)=0
\end{equation}
and usually have a complicated form.
In \cite{BB2}, various stability results 
\cite{YS,GD,W,SY,WL,ZW,ZY} for equation (\ref{2}) 
are compared  %\cite{YS,GD,W,SY,WL,ZW,ZY} 
using the test equation
\begin{equation}\label{3}
\dot{x}(t)+\sin^2 t \, [0.6 x(t-2)-bx(t)]=0, 
\end{equation}
where $b<0.6$ is a positive constant. This equation was considered first in the paper \cite{YS} for $b=\frac{1}{15}$,
and it was shown that the equation is asymptotically stable.   The best known 
condition $b<0.26$ for exponential stability of equation (\ref{3}) was obtained in \cite{BB2}.
In the present paper we improve all known stability tests for equation (\ref{3}), getting the estimate  $b<0.34$ for exponential stability.

Equation  (\ref{3}) is a special case of the equation with positive and negative coefficients and a non-delay term
\begin{equation}\label{4}
\dot{x}(t)+r(t)[ax(h(t))-bx(t)]=0,
\end{equation}
which is also a particular case of the equation with two delays
\begin{equation}\label{5}
\dot{x}(t)+r(t)[ax(h(t))-bx(g(t))]=0,
\end{equation}
where $a>b>0$ are constant, $r(t)\geq 0$, while $h$ and $g$ are delayed arguments.

Equations (\ref{4}) and (\ref{5})  appear as linearizations for many mathematical models
including Mackey--Glass equation (\ref{51}) and its modifications. 
In the present paper we obtain explicit exponential stability conditions for (\ref{4}) and (\ref{5}), which are easy to 
verify.
In particular, under some natural additional conditions, equation (\ref{4}) is uniformly exponentially stable if
\begin{equation}\label{6}
(a+b)\limsup_{t\rightarrow\infty}\int_{h(t)}^t r(s)ds <1+\frac{1}{e} \, .
\end{equation}
Applying stability condition (\ref{6}) to equation  (\ref{3}), 
we obtain the estimate  $b<0.34$, as was mentioned above. 

We apply our results to Mackey--Glass models with two delays (\ref{51}) and (\ref{MG})
deducing explicit local exponential stability (LES) conditions for 
the positive equilibrium and illustrate these results with numerical simulations. 

To obtain stability results for linearized equations, we apply the following tools: 
\begin{itemize}
\item
Bohl--Perron theorem which reduces the exponential stability problem to the norm estimation 
for linear operators in some functional spaces on semi-axes; 
\item
various transformations of a given equation including a transformation
of the independent variable $t$; 
\item
properties of equations with a positive fundamental function;
\item
a priori estimates of solutions and their derivatives. 
\end{itemize}

The paper is organized as follows. 
Section 2 contains a review of auxiliary results which are instrumental in the future proofs. 
In Section 3 the main stability results of the paper are obtained.
Section 4 involves an extension of these results to some more general models, 
such as equations with several positive and negative delayed terms, 
integro-differential equations and equations with a distributed delay. 
In Section 5 we apply the results obtained to Mackey--Glass type equations.
Section 6 presents a brief discussion of the results, as well as suggests some projects for future research.

\section{Preliminaries}

We consider equation  (\ref{1}) under the following conditions:
\begin{itemize}
\item[(a1)] $a$ and $b$ are Lebesgue measurable essentially bounded functions on $[0,\infty)$, $a(t)\geq b(t)\geq 0$;
\item[(a2)]  the functions $h$ and $g$ are Lebesgue measurable  on $[0,\infty)$, and $0\leq$ \lb
$ t-h(t)\leq \tau$, $0\leq t-g(t)\leq \delta$ 
for some finite constants $\tau$ and $\delta$.
\end{itemize}
Together with equation (\ref{1}), we consider for any $t_0\geq 0$ the initial value problem
\begin{equation}\label{7}
\dot{x}(t)+a(t)x(h(t))-b(t)x(g(t))=f(t), \ t\geq t_0; \quad x(t)=\varphi(t), \ t\leq t_0,
\end{equation}
where 
%\\
\begin{itemize}
\item[(a3)] the right-hand side $f$ is a Lebesgue measurable essentially bounded function on $[t_0,\infty)$, 
the initial function $\varphi:[t_0-\max\{\delta,\tau\}, t_0] \rightarrow \RR$ is a Borel measurable  
and bounded function.
\end{itemize}

\begin{definition}
The {\it solution} of problem (\ref{7}) is a locally  absolutely continuous on $[t_0,\infty)$ 
function satisfying the equation almost everywhere (a.e.) for $t\geq t_0$ 
and the initial conditions for $t\in [t_0-\max\{\delta,\tau\}, t_0]$.
The {\it fundamental function} $X(t,s)$ is a solution of the problem
\begin{equation*}
%\label{8}
\dot{x}(t)+a(t)x(h(t))-b(t)x(g(t))=0, \ t\geq s; \quad x(t)=0, \ t<s, \quad x(s)=1.
\end{equation*}
\end{definition}

The following  result incorporates the solution representation.

\begin{lemma}[{\cite[Theorem 4.3.1]{AS}}]\label{lemma1}
%%\hfill 
The solution of problem {\rm(\ref{7})} exists, is unique and has the form 
\begin{equation*}
%\label{9}
x(t)=X(t,t_0)\varphi(t_0)+\int_{t_0}^t X(t,s) f(s)ds -\int_{t_0}^t X(t,s) [a(s)\varphi(h(s))-b(s)\varphi(g(s))]ds,
\end{equation*}
where we assume $\varphi(t)=0$ for $t>t_0$.
\end{lemma}

\begin{definition} 
We will say that equation (\ref{1}) is {\it uniformly exponentially stable} 
if there exist positive numbers $M$ and $\gamma$ such that 
the solution of problem (\ref{7}) with $f\equiv 0$ and an arbitrary $t_0 \geq 0$ has the estimate 
\begin{equation}\label{10a}
|x(t)|\leq M e^{-\gamma (t-t_0)} \sup_{t \in [t_0 - \max\{ \delta, \tau\},  t_0]}|\varphi(t)|, 
\quad t\geq t_0,
\end{equation}
where $M$ and $\gamma$ do not depend on either $t_0$ or $\varphi$.
The fundamental function $X(t,s)$ of equation (\ref{1}) {\it has an exponential estimate} if it satisfies
$$
|X(t,s)|\leq M_0 e^{-\gamma_0(t-s)}, \quad  t\geq s\geq 0
$$
for some positive numbers $M_0>0$ and $\gamma_0>0$. 

We will say that an equilibrium solution of either (\ref{51}) or (\ref{MG}) is {\it locally exponentially stable} (LES)
if the linearized equation around this equilibrium is uniformly exponentially stable.
\end{definition}

Existence of an exponential estimate for the fundamental function is equivalent \cite{BB1} 
to the exponential stability for equations with bounded delays (see~(a2)). 
Evidently we can shift the initial point to some  $t_1 \in (0,t_0)$ and 
obtain an exponential estimate for $X(t,s)$ with some other constants $M_1>0$, $\gamma_1>0$.

The following result is usually referred to as the Bohl--Perron principle.

\begin{lemma}[{\cite[Theorem 4.7.1]{AS}}]\label{lemma2}
Assume that the solution of problem {\rm(\ref{7})}, where $\varphi(t)=0$, 
$t\leq t_0$, is bounded on $[t_0,\infty)$ for any $f$ which is
essentially bounded on $[t_0,\infty)$. 
Then equation {\rm(\ref{1})} is uniformly exponentially stable.
\end{lemma}

\begin{remark}
\label{remark1}
The Bohl--Perron principle was stated above for equation (\ref{1}) with 
two delays, but it is valid for linear equations
with an arbitrary number of delays, for integro-differential equations, and for equations with  
a distributed delay.
\end{remark}

\begin{lemma}[{\cite[Lemmas 9,10]{BB5}}]
\label{lemma_shift}
%%\hfill 
Suppose that exponential estimate {\rm(\ref{10a})} is valid for a solution of {\rm(\ref{7})}, with $t_1$ 
instead of $t_0$ and
$M_1$ instead of $M$, and we have a bounded solution growth, i.e.~for any  $t_0 \leq t_1$, there is a number $A\geq 1$ not dependent on either $t_0$ or $t_1$
such that
\begin{equation}
\label{growth}
\sup_{t \in [t_0,t_1]} x(t) \leq A \sup_{t \in [t_0 - \max\{ \delta, \tau\},  t_0]}|\varphi(t)|.
\end{equation}
Then there exists $M>0$ such that a solution of {\rm(\ref{7})} satisfies {\rm(\ref{10a})} with the same 
$\gamma$.
\end{lemma}

\begin{remark}
\label{remark2}
It follows from Lemma~\ref{lemma_shift} that in Lemma~\ref{lemma2} we can consider
boundedness of solutions not for all $f$ which are essentially bounded on $[t_0,\infty)$ but only
for those that vanish on $[t_0,t_1)$ for any fixed $t_1>t_0$.
\end{remark}

Further, 
in addition to a possible shift of the initial point,
we apply an argument transformation in the 
exponential estimate.

\pb

\begin{lemma}
\label{lemma_p}
Let $p:[0,\infty)\to [0,\infty)$ be a continuous increasing function,   
\begin{equation}
\label{10abc2}
\lim_{t \to \infty} p(t)=\infty
\end{equation}
and $t_0 \geq \max\{ \delta,\tau \}$.
Consider a set of continuous functions $$x:[p(t_0-\max\{ \delta,\tau \}),\infty) \to {\mathbb R}$$  satisfying
\begin{equation}\label{10abc1}
|x(p(t))|\leq M_1 e^{-\gamma_1 (p(t)-p(t_0))} \sup_{t \in [p(t_0-\max\{ \delta,\tau \}),p(t_0)]}| x(t)|,\quad t\geq t_0,
\end{equation}
where $M_1>0$ and $\gamma_1>0$ are independent of $t_0 \geq 0$ and of the values of $x(t)$
for $t \in [p(t_0-\max\{ \delta,\tau \}),p(t_0)]$.

Then, for any $t_1 \geq p(t_0)$ and $\sigma>0$, there exists $\sigma_1>0$ such that \lb
$ \sup_{t \leq t_1} | x(t)|<\sigma_1$ implies  $|x(t)|<\sigma$ for $t\geq t_1$,  and
$ \lim_{t \to \infty} x(t)=0$.

If, in addition, there exist  $T>0$ and $A>0$ such that
\begin{equation}
\label{10abc3}
\limsup_{t,u \to \infty,~t>u} \frac{p(t)-p(u)}{t-u} <A, \quad
\liminf_{t \to \infty} [p(t+T)-p(t)]>0
\end{equation}
then $x(t)$ satisfies {\rm(\ref{10a})} with some positive constant instead of $\max\{ \delta,\tau \}$, for some $M>0$, $\gamma>0$ not depending on 
$t_0 \geq \max\{ \delta,\tau \}$ and on the values $\varphi(t)$ of~$x(t)$ for $t \leq t_0$.
\end{lemma}

\begin{proof}
Let us fix $\sigma>0$, choose $t_1 \geq p(t_0)$, 
assume that (\ref{10abc2}) and (\ref{10abc1})  hold, and 
also that $ \sup_{t \in [p(t_0-\max\{ \delta,\tau \}),p(t_1)]} | x(t)|<\sigma_1$, where 
$\sigma_1 = \frac{\sigma}{M_1}$. 
Then $$ \sup_{t \in [p(t_0-\max\{ \delta,\tau \}),p(t_0)]} | x(t)| 
\leq  \sup_{t \leq \in [p(t_0-\max\{ \delta,\tau \}),t_1]} | x(t)| < \sigma_1,$$
so inequality (\ref{10abc1}) implies
$$ 
|x(p(t))|\leq M_1 e^{-\gamma_1 (p(t)-p(t_0))} \sup_{t \in [p(t_0-\max\{ \delta,\tau \}),p(t_0)]} | x(t)| < M_1 \sigma_1=\sigma,
\quad t \geq t_0.
$$
%for any $t \geq t_0$. 

We have $ |x(p(t))|<\sigma$ for $t \geq t_0$, where by (\ref{10abc2}) 
and continuity, the function $p(t)$ takes all the values in $[p(t_0),\infty)$.
The function $p$ is monotone increasing, therefore $t \geq t_0$ yields that $p(t) \geq 
p(t_0)$, and $|x(t)|<\sigma$ for any $t \geq p(t_0)$, in particular, for $t \geq t_1$. 
Further, 
$ \lim_{t \to \infty}  e^{-\gamma_1 (p(t)-p(t_0))}=0$ and
(\ref{10abc2}) imply
$ \lim_{t \to \infty} x(t)=0$.

Next, let (\ref{10abc3}) hold.
By the second inequality in (\ref{10abc3}), there exist $t_1>0$ and $\varepsilon>0$ such that 
$p(t+T)-p(t)>\varepsilon$ for $t \geq t_1$. 
Hence
$$
p(t)-p(t_1) \geq \left[ \frac{t-t_1}{T} \right] \varepsilon > 
\varepsilon \left(  \frac{t-t_1}{T} -1 \right)=\frac{\varepsilon}{T} (t-t_1) - \varepsilon,
$$
where $[t]$ is the integer part of $t$. 

Without loss of generality, using the first inequality in (\ref{10abc3}), we assume
$$
\frac{p(t)-p(u)}{t-u} \leq A, \quad t>u \geq t_1.
$$
Substituting $u=t_1$, we have $p(t) \leq p(t_1) + A (t-t_1)$.
Denoting   $w=A(t-t_1)$, we notice that $ p(t)-p(t_1) \geq \frac{\varepsilon}{AT}w
-\varepsilon$, thus estimate (\ref{10abc1}) implies 
$$
x(p(t_1)+w)\leq M_1 \exp\left\{-\gamma_1 \left( \frac{\varepsilon}{AT} w - \varepsilon \right) \right\}  \sup_{t \in [p(t_1-\max\{ \delta,\tau \}), p(t_1)]} |x(t)|, \quad w>0,
$$
the inequality holds if the initial point $p(t_1)$ is substituted by any $t_2 \geq p(t_1)$.  
Note that $p(t_1-\max\{ \delta,\tau \})\geq p(t_1)-A \max\{ \delta,\tau \}$.
Thus
$$
|x(t)|\leq M e^{-\gamma (t-t_2)} \sup_{t \in [t_2-A \max\{ \delta,\tau \},t_2]} | x(t)|, 
$$
where 
$$
\gamma=\frac{\varepsilon}{AT} \gamma_1,\quad  M=M_1 e^{\gamma_1 \varepsilon}, \quad t_2 \geq p(t_1).
$$
Further, (\ref{10abc1}) yields that (\ref{growth}) holds for $t_1=t_2$ and $A=M_1$.
Thus, the fact that we can shift $t_2$ to any point $t_0 < t_2$ follows from Lemma~\ref{lemma_shift}. 
\qed\end{proof}

%%\hspace*{-5pt}
The following example illustrates the significance of condition
(\ref{10abc3}) in \lb
Lemma~\ref{lemma_p}.

\begin{example}
The equation
$$
\dot{x}(t)+\frac{1}{t} x(t)=0, \quad t\geq 1
$$
has the solution $x(t)= \frac{x(1)}{t}$ which tends to zero but has no exponential estimate.
% here the second inequality fails. 
However, after the substitution $p(t)=\ln t$, for which the second inequality in (\ref{10abc3}) fails,
we get the equation $\dot{x}(s)+x(s)=0$, its solution has the exponential estimate
$$x(s) = x(s_0)e^{-(s-s_0)}.$$
\end{example}

All assumptions, definitions and results formulated above for equation (\ref{1}) are naturally extended 
to other scalar linear delay differential equations investigated in the paper, without additional discussion.

Consider now a linear equation with a single delay and a non-negative coefficient
\begin{equation}\label{11a}
\dot{x}(t)+a(t)x(h_0(t))=0, \quad a(t)\geq 0,\quad  0\leq t-h_0(t)\leq \tau_0,
\end{equation}
and denote by $X_0(t,s)$ the fundamental function of equation (\ref{11a}).

\begin{lemma}[{\cite[Theorem 2.21]{ABBD}}]\label{lemma3}
Assume that $X_0(t,s)>0$, $t\geq s\geq t_0$. Then 
$$
\int_{t_0+\tau_0}^t X_0(t,s) a(s)ds\leq 1.
$$
\end{lemma}

\begin{lemma}[{\cite[Theorem 2.7]{BB3},\cite[Theorem 3.1.1]{GL}}]\label{lemma4}
If for some $t_0\geq 0$
$$
 \int_{\min\{t_0,h_0(t)\}}^t a(s) ds\leq \frac{1}{e},\quad t\geq t_0
$$
then $X_0(t,s)>0$, $t\geq s\geq t_0$.

If in addition $a(t)\geq a_0>0$ then equation {\rm(\ref{11a})} is 
uniformly exponentially stable.
\end{lemma}

To extend stability results obtained for equation (\ref{1}) to equations with more general delays,
we will need  the following three ``transformation" results reducing terms with either distributed or 
several concentrated delays to a single term with a concentrated delay.

\begin{lemma}[{\cite[Lemma 5]{JMAA_lin_nonlin}}]\label{lemma5}
Assume that $a_k(t), h_k(t), k=1,\dots,m$ are measurable functions, $a_k(t)\geq 0$, $h_k(t)\leq t$, $k=1,\dots, m$ and $x$ is continuous on $[t_0, \infty)$.
Then there exists a measurable function 
$h_0$ satisfying 
$$h_0(t)\leq t, \quad \min_k h_k(t)\leq h_0(t)\leq \max_k h_k(t)$$ 
such that $$\sum_{k=1}^m a_k(t) x(h_k(t))=\left(\sum_{k=1}^m a_k(t)\right) x(h_0(t)).$$
\end{lemma}

\begin{lemma}[{\cite{MCM2008}}]\label{lemma7}
Assume that $B(t,s)$ is a measurable non-decreasing in $s$ function,  $h(t)$ is a measurable function,
$h(t)\leq t$, and $x$ is continuous  on $[t_0, \infty)$. 
Then there exists a measurable function 
$h_0$, $h(t)\leq h_0(t)\leq t$
such that $$\int_{h(t)}^t x(s)d_s B(t,s)  =\left(\int_{h(t)}^t d_s B(t,s)\right) x(h_0(t)).$$
\end{lemma}

As a particular case of Lemma~\ref{lemma7}, we obtain the following result.

\begin{lemma}\label{lemma6} %\cite{MCM2008}
Assume that $A(t,s)\geq 0,$ $ h(t)\leq t,$ and $x$ is continuous  on $[t_0, \infty)$. 
Then there exists a measurable function $h_0$, $h(t)\leq 
h_0(t)\leq t$ such that $$\int_{h(t)}^t 
A(t,s)x(s)ds =\left(\int_{h(t)}^t A(t,s)ds \right)x(h_0(t)).$$
\end{lemma}

\section{Main results}

We assume in this section that conditions (a1)--(a3) hold for equation (\ref{1}), 
and corresponding conditions are satisfied for equations (\ref{4}) and (\ref{5}).

Equation (\ref{1}) is well studied for the case $h(t)\equiv t$. In particular, the following result is known.

vavava
\begin{theorem}[{\cite[Corollary 2.4]{BB3},\cite[Corollary 3.13]{BB1}}]\label{theorem1}
Assume that $a(t)\geq a_0>0,$ $ h(t)\equiv t$ and
$$
\limsup_{t\rightarrow\infty} \frac{b(t)}{a(t)}<1.
$$
Then equation {\rm(\ref{1})} is uniformly exponentially stable.
\end{theorem}

Let us fix an interval $I=[t_0,t_1]$, $t_1>t_0\geq 0$, and for any essentially  bounded on 
$[t_0,\infty)$ 
function define
$|f|_I=\esssup_{t\in I} |f(t)|$, $\|f\|_{[t_0,\infty)}=\esssup_{t\geq t_0} |f(t)|$, $a^+=\max \{ a,0 \}$.
We are in a position to state and prove the first main result of the present paper.

\begin{theorem}\label{theorem2}
Assume that 
\begin{equation}\label{cond1}
a(t)-b(t) \neq 0 \text{ almost everywhere {\rm(}a.e.{\rm)},} 
\quad \int_0^{\infty} [a(s)-b(s)] ds=\infty,
\end{equation}
and for some $t_0 \geq 0$
\begin{equation}
\label{star} 
%%\left. \begin{aligned}
\left\| \!\left( \!  \int_{h(\cdot)}^{ \, \cdot} \! \! [a(s) \!- \!b(s)]
ds \!- \!\frac{1}{e} \!\right)^{ \! \! \! +}\right\|_{[t_0,\infty)}
%%& \\
 \! \! \! \!+ \!2 \!\left\| \!\frac{b}{a \!- \!b} \right\|_{[t_0,\infty)} \!\left\| \int_{h(\cdot)}^{g(\cdot)} \! \! \! [a(s) \!- \!b(s)] ds 
\right\|_{[t_0,\infty)} \! \! \! \!
%%& \end{aligned}\right\}
< \!1.
\end{equation} 
%where all the functions and norms are considered for $t \geq t_0$ for some $t_0\geq 0$.
Then equation {\rm(\ref{1})}
is asymptotically stable.
If in addition there exists $T>0$ such that
\begin{equation}\label{cond1abc}
\liminf_{t \to\infty} \int_{t}^{t+T} [a(s)-b(s)] ds >0
\end{equation}
then {\rm(\ref{1})}
is uniformly exponentially stable.
\end{theorem}
\begin{proof}
Rewrite equation (\ref{1}) as
\begin{equation}\label{1a}
\dot{x}(t)+[a(t)-b(t)]x(h(t))+b(t)[x(h(t))-x(g(t))]=0
\end{equation} %Seite310
and denote $$s=p(t):=\int_{t_0}^t [a(\zeta)-b(\zeta)] d\zeta,$$ where $p(t_0)=0$. 
By the assumptions of the theorem, $p$ is a strictly monotone increasing 
function satisfying $\lim_{t\rightarrow\infty} p(t)=\infty$. Let us make the substitution 
$t=p^{-1}(s)$, $x(t)=y(s)$, then 
$$ \dot{x}(t)=[a(p^{-1}(s))-b(p^{-1}(s))]\dot{y}(s),\quad
x(h(t))=y(h_0(s)), \quad x(g(t))=y(g_0(s)),
$$
and
$$
s-h_0(s)=\int_{h(t)}^t [a(\zeta)-b(\zeta)] d\zeta, \quad
s-g_0(s)=\int_{g(t)}^t [a(\zeta)-b(\zeta)] d\zeta.
$$
Denote 
$$
\tau_0=\esssup_{t\geq t_0} \int_{h(t)}^{t} [a(\zeta)-b(\zeta)] d\zeta,\quad 
\delta_0=\esssup_{t\geq t_0} \int_{g(t)}^{t} [a(\zeta)-b(\zeta)] d\zeta.
$$
Hence $s-h_0(s)\leq \tau_0<\infty$, $s-g_0(s)\leq \delta_0<\infty$.
Equation  (\ref{1a}) has the form
\begin{equation}\label{11}
\dot{y}(s)+y(h_0(s))+\frac{b(p^{-1}(s))}{a(p^{-1}(s))-b(p^{-1}(s))}\int_{g_0(s)}^{h_0(s)} \dot{y}(\zeta)d\zeta=0.
\end{equation}

To prove asymptotic stability
of equation (\ref{11}), consider the initial value problem
\begin{equation}\label{12}
\begin{alignedat}{2}
 \dot{y}(s)+y(h_0(s))+\frac{b(p^{-1}(s))}{a(p^{-1}(s))-b(p^{-1}(s))}\int_{g_0(s)}^{h_0(s)} 
\dot{y}(\zeta)d\zeta&=f(s),& \quad s&\geq 0,\\
y(s)=\dot{y}(s)&=0,&  s&\leq 0,
\end{alignedat}
\end{equation}
where $f$ is an essentially bounded function on $[0,\infty)$ such that 
\begin{equation}\label{add}
f(s)=0, \quad s\leq s_0=\max\left\{ \tau_0,\delta_0,\frac{1}{e} \right\}.
\end{equation}
%$$
%\tau_0=\sup_{s\geq 0} \int_{h(p^{-1}(s))}^{p^{-1}(s)} [a(\varsigma)-b(\varsigma)] d\varsigma,~
%\delta_0=\sup_{s\geq 0} \int_{g(p^{-1}(s))}^{p^{-1}(s)} [a(\varsigma)-b(\varsigma)] d\varsigma.
%$$
Conditions (\ref{12}) and (\ref{add}) imply that the solution of (\ref{12})
satisfies $y(s)=$ \lb
$\dot{y}(s)=0,$ $s\leq s_0$.

Denote
$$
r_0(s)=\left\{\begin{array}{ll}
h_0(s), &  h_0(s)\geq s-\frac{1}{e},  \vspace{2mm} \\
 s-\frac{1}{e},&  h_0(s) < s-\frac{1}{e},\\
\end{array}\right.
$$
then $s-r_0(s)\leq \frac{1}{e}$. We can rewrite equation (\ref{12}) as
\begin{equation}\label{13}
%%\begin{split}
\dot{y}(s) \!+ \!y(r_0(s))
%%\\ &
 \!= \! - \!\int_{r_0(s)}^{h_0(s)} \! \! \! \dot{y}(\zeta)d\zeta \!- \! 
\frac{b(p^{-1}(s))}{a(p^{-1}(s)) \!- \!b(p^{-1}(s))} \!\int_{g_0(s)}^{h_0(s)} \! \! \!
\dot{y}(\zeta)d\zeta \!+ \!f(s).
%%\end{split}
\end{equation}
Let $Y_0(s,\zeta)$ be the fundamental function of the equation
\begin{equation}\label{3.8a}
\dot{y}(s)+y(r_0(s))=0.
\end{equation}
By Lemma \ref{lemma4} we have $Y_0(s,\zeta)>0$, and equation (\ref{3.8a}) is uniformly exponentially stable.

From (\ref{13}) and Lemma \ref{lemma1}, we get
\begin{equation}\label{add2}
\begin{split}
y(s) \!=& \!  - \! \!  \int_{0}^s \! \! Y_0(s,\zeta) \! \left[ \! \int_{r_0(\zeta)}^{h_0(\zeta)} \! \! \! %%\hspace*{-7pt}
\dot{y}(\xi)d\xi
%%\right. \\  
 \!+ \! 
%%\left.
\frac{b(p^{-1}(\zeta))}{a(p^{-1}(\zeta)) \!- \!b (p^{-1}(\zeta))} \!\int_{g_0(\zeta)}^{h_0(\zeta)} \! \! 
\dot{y}(\xi)d\xi \!\right] \!d\zeta \\
&+ \!f_1(s),
\end{split}
\end{equation}
where $f_1(s)=\int_{0}^s Y_0(s,\zeta) f(\zeta) d\zeta$. Since $Y_0(s,\tau)$ has an exponential estimate, 
$\|f_1\|_{[0,\infty)}<\infty$.

Since the right-hand side of (\ref{13}) is equal to zero for $s\leq s_0$, where $s_0\geq \frac{1}{e}$,
the zero lower bound in the first integral in (\ref{add2}) can be replaced with $s_0$.

In the following, up to the end of the proof, we omit the index $[t_0,\infty)$ in the norm of 
the functions on $[t_0,\infty)$. 
Let us fix an interval $I=[0,s_1]$. 
By Lemma~\ref{lemma3}, we have $ 0 \leq \int_{s_0}^s Y_0(s,\zeta)  d\zeta \leq 1$, thus
\begin{equation}
\begin{split}
|y|_I & \leq  \left[ \left\| (r_0-h_0)^+ \right\|
+\left\|\frac{b}{a-b}\right\|\|g_0-h_0\|\right]|\dot{y}|_I+\|f_1\|,
%%\nonumber
\\
(r_0(s)-h_0(s))^+ & =  \left( (s-h_0(s))-(s-r_0(s)) \right)^{ \! +} \\
&=\left(\int_{h(t)}^t [a(\zeta)-b(\zeta) ]d\zeta-\frac{1}{e}\right)^{ \! +},
%%\nonumber
\\
g_0(s)-h_0(s) & =  (s-h_0(s))-(s-g_0(s))
%%\nonumber
\\
 & =  \int_{h(t)}^t [a(\zeta)-b(\zeta) ]d\zeta-\left(\int_{g(t)}^t [a(\zeta)-b(\zeta) ]d\zeta\right) \label{add_star}
%\nonumber
\\ & = 
\int_{h(t)}^{g(t)} [a(\zeta)-b(\zeta) ]d\zeta. 
%%\nonumber
\end{split}
\end{equation}
Hence
\begin{equation}
\begin{split}
|y|_I & \leq  \left[\left\|\left( \int_{h(t\cdot)}^{ \, \cdot} [a(\zeta)-b(\zeta) ]d\zeta-\frac{1}{e}\right)^{ \! +}\right\| \right.
\\%\nonumber \\ 
&\quad+\left.\left\|\frac{b}{a-b}\right\|\left\| ~ \int_{h(\cdot)}^{g(\cdot)} [a(\zeta)-b(\zeta) ]d\zeta 
\right\|\right]|\dot{y}|_I+\|f_1\|.
\end{split}\label{14}
\end{equation}

From equality (\ref{12}) and the last part of (\ref{add_star}), we have 
$$
|\dot{y}|_I\leq |y|_I+\left\|\frac{b}{a-b}\right\|\left\| \int_{h(\cdot)}^{g(\cdot)} 
[a(\zeta)-b(\zeta) ]d\zeta\right\||\dot{y}|_I +\|f\|.
$$
Therefore
\begin{equation}\label{15}
|\dot{y}|_I \leq \frac{1}{1-\left\|\frac{b}{a-b}\right\|\left\| \int_{h(\cdot)}^{g(\cdot)} 
[a(\zeta)-b(\zeta) ]d\zeta\right\|}|{y}|_I+M_1,
\end{equation}
where the denominator is positive by (\ref{star}) and
$$
M_1=\frac{\|f\|}{1-\left\|\frac{b}{a-b}\right\|\left\| \int_{h(\cdot)}^{g(\cdot)} [a(\zeta)-b(\zeta) ]d\zeta\right\|}  .
$$
Inequalities  (\ref{14}) and  (\ref{15})  imply 
\begin{equation}\label{16}
|{y}|_I \! \leq \! \frac{\left\| \!\left( \!\int_{h(\cdot)}^{ \, \cdot} [a(\zeta) \!- \!b(\zeta) ]d\zeta \!- \!\frac{1}{e} \!\right)^{  \!+}\right\| \!
+ \!\left\|\frac{b}{a-b}\right\| \!\left\| \! \int_{h(\cdot)}^{g(\cdot)} [a(\zeta) \!- \!b(\zeta) ]d\zeta\right\|}
{1 \!- \!\left\|\frac{b}{a-b}\right\| \!\left\| \int_{h(\cdot)}^{g(\cdot)} [a(\zeta) \!- \!b(\zeta) ]d\zeta\right\|}|{y}|_I \!+ \!M_2,
\end{equation}
where
\begin{align*}
M_2 & =  \frac{\left\|\left(\int_{h(\cdot)}^{ \, \cdot} [a(\zeta)-b(\zeta) ]d\zeta-\frac{1}{e}\right)^{ \! +} \right\|
+\left\|\frac{b}{a-b}\right\|\left\|\int_{h(\cdot)}^{g(\cdot)} [a(\zeta)-b(\zeta) ]d\zeta\right\|}
{1-\left\|\frac{b}{a-b}\right\|\left\|\int_{h(\cdot)}^{g(\cdot)} [a(\zeta)-b(\zeta) ]d\zeta\right\|} M_1   \\ &\quad+ ~
\frac{\| f_1 \|} {1-\left\|\frac{b}{a-b}\right\|\left\|\int_{h(\cdot)}^{g(\cdot)} [a(\zeta)-b(\zeta) ]d\zeta\right\|}.
\end{align*}
Inequality (\ref{16}) has the form $|{y}|_I \leq  \alpha |{y}|_I+M_2$, where the numbers $\alpha>0$
and $M_2>0$ do not depend on the interval $I$. Inequality  (\ref{star}) implies  $\alpha<1$ and thus $|{y}|_I \leq 
\frac{M_2}{1-\alpha}$. 
Therefore for any essentially bounded function $f$ on $[t_0,\infty)$ (vanishing on  $[t_0,t_1]$ for some 
$t_1> t_0$), the solution of  problem (\ref{12}) is bounded on $[t_0,\infty)$. Thus by Lemma \ref{lemma2} equation (\ref{11})
is uniformly exponentially stable. 

Hence for the fundamental function $Y(s,l)$ of equation (\ref{11}) there exist $\lambda>0$ and  $M>0$ such that
$$
|Y(s,l)|\leq M e^{-\lambda (s-l)},\quad s\geq l\geq 0,
$$
and its solution $y(s)$ has an exponential estimate.
Since $y(s)=x(p(t))$, by (\ref{cond1}) and 
Lemma~\ref{lemma_p}, equation (\ref{1}) is asymptotically stable. 
Also, by Lemma~\ref{lemma_p}, under (\ref{cond1abc}), where  (\ref{cond1abc}) and global essential boundedness 
of $a$ and $b$ imply (\ref{10abc3}), equation (\ref{1}) is also uniformly exponentially stable. 
\qed\end{proof}

\begin{corollary}\label{corollary1}
Assume that {\rm(\ref{cond1})} is satisfied, and for some $t_0\geq 0$ one of the following two conditions holds:
\begin{itemize}
\item[\rm 1)]
\begin{align}\label{17}
\esssup_{t\geq t_0} \int_{h(t)}^t [a(s)-b(s)]ds &\leq \frac{1}{e}, \\
\left\|\frac{b}{a-b}\right\|_{[t_0,\infty)}\left\| ~ \int_{h(\cdot)}^{g(\cdot)}
 [a(\zeta)-b(\zeta) ]d\zeta\right\|_{[t_0,\infty)} &< \frac{1}{2};
\nonumber
\end{align}
\item[\rm 2)]
\begin{equation}\label{18}
\esssup_{t\geq t_0} \int_{h(t)}^t [a(s)-b(s)]ds > \frac{1}{e},~
\end{equation}
\[
%%\begin{split}
\left\|\int_{h(\cdot)}^{ \, \cdot} \! \! [a(s) \!- \!b(s)]ds \right\|_{[t_0,\infty)}
 \! \! \!+ \!2 \!\left\| \!\frac{b}{a \!- \!b}\right\|_{[t_0,\infty)} \!\left\| \int_{h(\cdot)}^{g(\cdot)}
 \! \! [a(\zeta) \!- \!b(\zeta) ]d\zeta\right\|_{[t_0,\infty)} \! \!  \!
%%\\ &
< \!1 \!+ \!\frac{1}{e}
%%\end{split}
\]
\end{itemize}
Then equation {\rm (\ref{1})} is asymptotically stable. 
If in addition there exists $T>0$ such that {\rm(\ref{cond1abc})} holds,
{\rm(\ref{1})} is uniformly exponentially stable.
\end{corollary}

\begin{proof}
Since for any essentially bounded function $f$ and a number $c\geq 0$
we have
$$
\esssup_{t\geq 0}(f(t)-c)^+=\left\{
\begin{array}{lll}
 \!0,& \mbox{if}& \ \esssup_{t\geq 0} f(t)\leq c,\\
 \!\esssup_{t\geq 0} f(t)-c, \ & \mbox{if} & \ \esssup_{t\geq 0} f(t)>c,\\
\end{array}\right.
$$
we obtain the statement of the corollary by applying (\ref{star}) in Theorem~\ref{theorem2} in the 
two cases.
\qed\end{proof}

\begin{corollary}\label{corollary2}
Suppose that $g(t)\equiv t$, {\rm(\ref{cond1})} is satisfied, and for some $t_0\geq 0$ one of the following two assumptions 
holds: %\vspace{2mm}
\begin{itemize}
\item[\rm 1)] condition {\rm(\ref{17})} and $$ \left\|\frac{b}{a-b}\right\|_{[t_0,\infty)}\left\|  \int_{h(\cdot)}^{ \, \cdot} [a(\zeta)-b(\zeta) 
]d\zeta\right\|_{[t_0,\infty)}<\frac{1}{2};$$ 
\item[\rm 2)] condition {\rm(\ref{18})} and $$ \left(1+2\left\|\frac{b}{a-b}\right\|_{[t_0,\infty)}\right)
\left\| ~ \int_{h(\cdot)}^{ \, \cdot} [a(\zeta)-b(\zeta) ]d\zeta\right\|_{[t_0,\infty)}<1+\frac{1}{e}.$$
\end{itemize}
Then equation {\rm(\ref{1})} is asymptotically stable. If in addition there exists $T>0$ such that {\rm(\ref{cond1abc})} holds,
{\rm(\ref{1})} is uniformly exponentially stable.
\end{corollary}

Consider now equation (\ref{5}), where 
\begin{equation}\label{19a}
a>b>0, \quad r(t)\geq 0, \quad \int_{0}^{\infty} r(s) ds=\infty, \quad r(t)\neq 0 \ \mbox{a.e.}
\end{equation}

\begin{corollary}\label{corollary3}
Assume that {\rm(\ref{19a})} and for some $t_0\geq 0$ one of the following conditions hold:
%\vspace{2mm}
\begin{itemize}
\item[\rm 1)]
$ 
(a-b)\esssup_{t\geq t_0} \int_{h(t)}^t r(s)ds \leq \frac{1}{e},\quad
b \, \esssup_{t\geq t_0} \left|  \int_{h(t)}^{g(t)}r(s)ds\right|<\frac{1}{2};
$
\item[\rm 2)] $
(a-b)\esssup_{t\geq t_0} \int_{h(t)}^t r(s)ds > \frac{1}{e},
$ \\
$
(a-b)\esssup_{t\geq t_0} \int_{h(t)}^t r(s)ds +2b \, \esssup_{t\geq t_0} 
\left|\int_{h(t)}^{g(t)}r(s)ds\right|<1+\frac{1}{e}.
$
\end{itemize}
Then equation {\rm (\ref{5})} is asymptotically stable. 
If in addition there exists $T>0$ such that
\begin{equation}\label{cond1abcd}
\liminf_{t \to\infty} \int_{t}^{t+T} r(s)~ ds >0
\end{equation}
then {\rm (\ref{5})} is uniformly exponentially stable.
\end{corollary}

\begin{corollary}\label{corollary4}
Assume that {\rm (\ref{19a})} is satisfied,  $g(t)\equiv t$, and for some $t_0\geq 0$ one of the following conditions holds:
\begin{itemize}
\item[\rm 1)] $ \esssup_{t\geq t_0}\int_{h(t)}^t r(s)ds \leq \frac{1}{e(a-b)}, \quad
\esssup_{t\geq t_0} \int_{h(t)}^t r(s)ds < \frac{1}{2b};$
\item[\rm 2)] $ \esssup_{t\geq t_0} \int_{h(t)}^t r(s)ds >\frac{1}{e(a-b)},\quad
(a+b)\esssup_{t\geq t_0} \int_{h(t)}^t r(s)ds<1+\frac{1}{e}.
$
\end{itemize}
Then equation {\rm (\ref{4})} is  asymptotically stable.
If in addition there exists $T>0$ such that {\rm (\ref{cond1abcd})}
holds then  {\rm (\ref{4})} is
uniformly exponentially stable. 
\end{corollary}

\begin{example}\label{example1}
Consider  test equation (\ref{3}). We will estimate 
the values of the parameter $b$ for which the condition 2) of 
Corollary~\ref{corollary4} holds. Here $a=0.6$ and $r(t)=\sin^2 t$.
We have
$
\int_{t}^{t+2} \! \sin^2 s \, ds\geq \frac{1}{2},
$ 
hence condition (\ref{cond1abcd}) holds with $T=2$. Also
\[
%%\begin{split}
\int_{t}^{t+2} \! \! \!\!\! \sin^2 \! \! s \, ds \!= \! \!
1 \!+ \!\frac{1}{4}[\sin(2t) \!- \!\sin(2(t \!- \!2) \!)] \!= \! \!1 \!+ \!\frac{1}{2}\sin 2 \cos(2(t \!- \!1) \!)
%%\\ &
 \! \!\leq  \!\! 1 \!+ \!\frac{1}{2} \sin 2 \! \! < \! \! 1.4547.
%%\end{split}
\]
We easily verify that Part 1 of Corollary~\ref{corollary4} cannot be applied. 
The first inequality in Part 2 of Corollary~\ref{corollary4} is 
$$ b< 0.6 - \frac{1}{1.4547 e} \approx 0.3471,$$ 
while the second inequality 
$ 
(0.6+b) 1.4547<1+\frac{1}{e} $
implies
$$
b<\frac{1+\frac{1}{e}}{1.4547}-0.6\approx 0.3403.
$$
%Since $\sup_{t\geq 0}\int _{t-2}^t \sin^2 s ~ ds\approx  1.4546$, the first inequality in 2)  
%holds for $b<0.3471$ and the 
%second one holds for 
%$b<0.3403$. 
Hence equation (\ref{3}) is uniformly exponentially stable for $b<0.34$. We recall that the best known estimate \cite{BB2}
was $b<0.26$.
\end{example}

Below we present the next main result of the paper.

\begin{theorem}\label{theorem3}
Assume that 
\begin{equation}\label{cond2}
a(t)\neq 0 \ \mbox{a.e.}, \quad \int_0^{\infty} a(s) ds=\infty 
\end{equation}
and for some $t_0\geq 0$
\begin{equation}\label{18abc}
\esssup_{t\geq t_0}\frac{b(t)}{a(t)}<1, \quad \esssup_{t\geq t_0} \left|  \int_{h(t)}^{g(t)} a(s) ds\right|<1,
\end{equation}
%and for some $t_0\geq 0$ 
\begin{equation}\label{18a}
\hspace*{-6pt}\left\|\left(  \int_{h(\cdot)}^{ \, \cdot} a(s) ds-\frac{1}{e}\right)^{ \! +}\right\|_{[t_0,\infty)}
<\frac{1-\left\|\frac{b}{a}\right\|_{[t_0,\infty)}}{\left\|1-\frac{b}{a}\right\|_{[t_0,\infty)}}
\left(1-\left\|  \int_{h(\cdot)}^{g(\cdot)} a(s) ds\right\|_{[t_0,\infty)}\right)
\end{equation}
Then equation {\rm(\ref{1})} is asymptotically stable.
If in addition there exists $T>0$ such that the condition
\begin{equation}\label{18b}
\liminf_{t\rightarrow\infty}\int_{t}^{t+T} a(s)ds>0
\end{equation}
 holds then {\rm(\ref{1})} is uniformly exponentially stable.
\end{theorem}
\begin{proof}
We proceed similarly to the proof of Theorem~\ref{theorem2}.
Denote $s=p(t):=\int_{t_0}^t a(\zeta) d\zeta$, $p(t_0)=0$. 
By the conditions on $a$, the function $p$ is strictly monotone 
increasing and $\lim_{t\rightarrow\infty} p(t)=\infty$. After the substitution 
$t=p^{-1}(s)$, $x(t)=y(s)$ we have 
$$
\dot{x}(t)=a(p^{-1}(s))\dot{y}(s), \quad x(h(t))=y(h_0(s)), \quad x(g(t))=y(g_0(s)),
$$
where 
$$
s-h_0(s)=\int_{h(t)}^t a(\zeta) d\zeta,  \quad s-g_0(s)=\int_{g(t)}^t a(\zeta) d\zeta,
$$
and equation  (\ref{1}) has the form
\begin{equation}\label{19}
\dot{y}(s)+y(h_0(s))-\frac{b(p^{-1}(s))}{a(p^{-1}(s))} y(g_0(s))=0
\end{equation}
In order to prove asymptotic stability of 
equation (\ref{19}), consider the initial value problem
\begin{equation}\label{20}
\begin{alignedat}{2}
\dot{y}(s)+y(h_0(s))-\frac{b(p^{-1}(s))}{a(p^{-1}(s))}y(g_0(s))&=f(s),&\quad s&\geq 0,\\
y(s)=\dot{y}(s)&=0,& ~s&\leq 0,
\end{alignedat}
\end{equation}
where $f$ is an essentially bounded function on $[s_0,\infty)$ such that 
$$f(s)=0, \quad s\leq s_0=\max\left\{ \tau_0,\delta_0,\frac{1}{e}
  \right\},$$
$$\tau_0=\esssup_{t \geq t_0} \int_{h(t)}^t a(\varsigma) d\varsigma,\quad
\delta_0=\esssup_{t \geq t_0} \int_{g(t)}^t a(\varsigma) d\varsigma.$$

Denote
$$
r_0(s)=\left\{ \! \begin{array}{ll}
h_0(s), &   h_0(s) >  s-\frac{1}{e}, \vspace{2mm} \\
s-\frac{1}{e}, \ &  h_0(s)\leq s-\frac{1}{e}.\\
\end{array}\right.
$$
We have $s-r_0(s)\leq \frac{1}{e}$. Equation (\ref{20}) can be rewritten as
\begin{equation}\label{21}
\dot{y}(s)+y(r_0(s))= \int_{h_0(s)}^{r_0(s)} \dot{y}(\zeta)d\zeta+
\frac{b(p^{-1}(s))}{a(p^{-1}(s))}y(g_0(s))+f(s).
\end{equation}
Let $Y_0(s,\zeta)$ be the fundamental function of the equation
\begin{equation}\label{12a}
\dot{y}(s)+y(r_0(s))=0.
\end{equation}
By Lemma \ref{lemma4}, we get that $Y_0(s,\tau)>0$ and equation (\ref{12a}) is uniformly exponentially stable.
From (\ref{21}) and Lemma \ref{lemma1}, we have
$$
y(s)=\int_{s_0}^s Y_0(s,\zeta) \left[ \int_{h_0(\zeta)}^{r_0(\zeta)} \dot{y}(\xi)d\xi
+
\frac{b(p^{-1}(\zeta))}{a(p^{-1}(\zeta))}y(g_0(\zeta))\right]d\zeta+f_1(s),
$$
where $f_1(s)=\int_{s_0}^s Y_0(s,\zeta) f(\zeta) d\zeta$. 
Since $Y_0(s,\zeta)$ has an exponential estimate, $\|f_1\|_{[s_0,\infty)}<\infty$.

Further we omit the index $[t_0,\infty)$ in the norm of %Seite317
the functions on $[t_0,\infty)$  and assume $I=[s_0,s_1]$, where $s_1>s_0$ is fixed.
By Lemma \ref{lemma3}, 
$$
|y|_I\leq \|(r_0-h_0)^+\| \, |\dot{y}|_I+\left\|\frac{b}{a}\right\| \, |y|_I+\|f_1\|.
$$
Also,
\begin{gather*}
(r_0(s) \!- \!h_0(s))^+ \! = \!  (s \!- \!h_0(s)) \!- \!(s \!- \!r_0(s))^+ \!= \!\left(\int_{h(t)}^t \! \! \! a(\zeta)d\zeta \!- \!\frac{1}{e} \!\right)^{ \! +} \! \!,
\\
g_0(s) \!- \!h_0(s) \!  = \!  (s \!- \!h_0(s)) \!- \!(s \!- \!g_0(s))
%%\\ & 
 \!=  \! \! \int_{h(t)}^t \! \! \! a(\zeta)d\zeta \!- \! \!\int_{g(t)}^t \! \! \! a(\zeta)d\zeta \!= \!
 \!\int_{h(t)}^{g(t)} \! \! \! a(\zeta)d\zeta.
\end{gather*}

Therefore
\begin{equation}\label{22}
|y|_I\leq\frac{ \|(r_0-h_0)^+\|}{1-\left\|\frac{b}{a}\right\|}|\dot{y}|_I+M_1, \quad 
\mbox{where} \quad M_1 := \frac{\|f_1\|}{1-\| \frac{b}{a} \|}.
\end{equation}
Rewriting equation  (\ref{20}) as 
$
\dot{y}(s) \!= \!\int_{h_0(s)}^{g_0(s)}\dot{y}(\zeta) d\zeta \!- \!\left( \!1 \!- \!\frac{b(p^{-1}(s))}{a(p^{-1}(s))} \! \right)  \!
y(g_0(s))+ \!f(s)
$
implies
$$
|\dot{y}|_I\leq \|h_0-g_0\| \, |\dot{y}|_I+\left\|1-\frac{b}{a}\right\||y|_I+\|f\|.
$$
Then 
\begin{equation}\label{23}
|\dot{y}|_I \leq \frac{\left\|1-\frac{b}{a}\right\|}{1- \|h_0-g_0\|}|y|_I+M_2, \quad \mbox{where} \quad M_2:=\frac{\| f\|}{1- 
\|h_0-g_0\|} \, .
\end{equation}
Inequalities  (\ref{22}) and  (\ref{23})  yield that 
\begin{equation}\label{24}
%%\begin{split}
|{y}|_I \! \!\leq \! \! \frac{ \|(r_0 \!- \!h_0)^{ \! +} \!\|}{1 \!- \!\left\|\frac{b}{a}\right\|}
\frac{\left\|1 \!- \!\frac{b}{a}\right\|}{1 \!- \! \|h_0 \!- \!g_0\|}|y|_I  \!+ \!M,\ \
 \mbox{where} \ \
M \! \! := \! \frac{ \|(r_0 \!- \!h_0)^{ \! +} \!\|}{1 \!- \!\left\|\frac{b}{a}\right\|} M_2 \!+ \!M_1.
%%\end{split}
\end{equation}
We have
$$
 \|(r_0(s)-h_0(s))^+\|=\left\|\left(  \int_{h(\cdot)}^{ \, \cdot} a(s) ds-\frac{1}{e}\right)^{ \! +}\right\|, \quad
\|h_0-g_0\|=\left\|  \int_{h(\cdot)}^{g(\cdot)} a(s) ds\right\|.
$$
Inequality (\ref{18a}) implies 
$
\frac{ \|(r_0(s)-h_0(s))^+\|}{1-\left\|\frac{b}{a}\right\|}\frac{\left\|1-\frac{b}{a}\right\|}{1- \|h_0-g_0\|}<1.
$
From  (\ref{24}) we have 
\begin{equation}\label{25}
|{y}|_I \leq \frac{M}{1- \frac{\|(r_0-h_0)^+\| 
\left\|1-\frac{b}{a}\right\|}{\left(1-\left\|\frac{b}{a}\right\|\right) (1- \|h_0-g_0\|)}} 
\end{equation}

The right-hand side of (\ref{25}) does not depend on the interval $I$. Hence
for any  bounded on $[s_0,\infty)$ function $f$, the solution of  problem (\ref{20})
is a  bounded on $[s_0,\infty)$ function. 
Then, by Lemma \ref{lemma2} equation (\ref{19}) is uniformly exponentially stable. 

However $y(s)=x(p(t))$,  therefore  (\ref{cond2}) and
Lemma~\ref{lemma_p} yield that equation (\ref{1}) is asymptotically stable.

Since (\ref{18b}), together with boundedness of $a$,  implies (\ref{10abc3}),
by Lemma~\ref{lemma_p}, under (\ref{18b})    
equation (\ref{1}) is also uniformly exponentially stable.
\qed\end{proof}

\begin{corollary}\label{corollary5}
Assume that {\rm(\ref{cond2})} and {\rm(\ref{18abc})} are satisfied  
and for some $t_0\geq 0$ one of the following conditions holds:
\begin{itemize}
\item[\rm 1)] $
\esssup_{t\geq t_0} \int_{h(t)}^{t} a(s) ds \leq \frac{1}{e};
$
\item[\rm 2)] $
\esssup_{t\geq t_0} \int_{h(t)}^{t} a(s) ds >\frac{1}{e}
$ \quad and
$$
\esssup_{t\geq t_0}\int_{h(t)}^{t} a(s) ds<\frac{1-\esssup_{t\geq t_0} 
\frac{b(t)}{a(t)}}{1-\essinf_{t\geq t_0}  \frac{b(t)}{a(t)} }\left(1-\esssup_{t\geq t_0} 
\left|\int_{h(t)}^{g(t)} a(s) ds\right|\right)+\frac{1}{e}  .
$$
\end{itemize}
Then equation {\rm (\ref{1})}
is asymptotically stable. If in addition {\rm(\ref{18b})} holds for some $T>0$, {\rm(\ref{1})}
is uniformly exponentially stable. 
\end{corollary}

\begin{corollary}\label{corollary6}
Assume that $g(t)\equiv t$,  {\rm(\ref{cond2})} is satisfied,  for some $t_0\geq 0$ $$\esssup_{t\geq t_0}\frac{b(t)}{a(t)}<1,\quad
\esssup_{t\geq t_0} \int_{h(t)}^t a(s) ds < 1,
$$
and one of the following conditions holds:

\pb
\begin{itemize}
\item[\rm 1)] $
\esssup_{t\geq t_0} \int_{h(t)}^{t} a(s) ds \leq \frac{1}{e} \, ;
$
\item[\rm 2)] $
\esssup_{t\geq t_0}\int_{h(t)}^{t} a(s) ds> \frac{1}{e}
$  \quad and 
$$
\esssup_{t\geq t_0} \int_{h(t)}^{t} a(s) ds<\frac{1-\esssup_{t\geq t_0} 
\frac{b(t)}{a(t)}}{1-\essinf_{t\geq t_0} \frac{b(t)}{a(t)}} \left(1-\esssup_{t\geq t_0}
\int_{h(t)}^{t} a(s) 
ds\right)+\frac{1}{e}.
$$
\end{itemize}
Then equation {\rm(\ref{1})} is asymptotically stable. If in addition {\rm(\ref{18b})} holds for some $T>0$, {\rm(\ref{1}) }
is uniformly exponentially stable. 
\end{corollary}

Consider now equation (\ref{5}), where (\ref{19a}) holds.

\begin{corollary}\label{corollary7}
Let {\rm(\ref{19a})} be satisfied, for some $t_0\geq 0$
$$ 
\esssup_{t\geq t_0} \left|\int_{h(t)}^{g(t)} r(s) ds\right|<\frac{1}{a},$$ 
and one of the following conditions holds:
\begin{itemize}
\item[\rm 1)] $
a\esssup_{t\geq t_0} \int_{h(t)}^{t} r(s) ds \leq \frac{1}{e}
$;
\item[\rm 2)] $
\frac{1}{e} <  a\esssup_{t\geq t_0}\int_{h(t)}^{t} r(s) ds <  1
$  \quad and 
$$
a\left(\esssup_{t\geq t_0} \int_{h(t)}^{t} r(s) ds+
\esssup_{t\geq t_0} \left|\int_{h(t)}^{g(t)} r(s) ds\right|\right)<1+\frac{1}{e}.
$$
\end{itemize}
Then equation {\rm(\ref{5})}  is asymptotically stable.  
 If in addition there exists $T>0$ such that {\rm(\ref{cond1abcd})} holds
then {\rm(\ref{5})} is uniformly exponentially stable. 
\end{corollary}

\begin{corollary}\label{corollary8}
Assume that for some $t_0\geq 0$ one of the following conditions holds:
\begin{itemize}
\item[\rm 1)] $ \esssup_{t\geq t_0} \int_{h(t)}^{t} r(s) ds \leq \frac{1}{ae}
$;
\item[\rm 2)] $ \frac{1}{ae}< \esssup_{t\geq t_0} \int_{h(t)}^{t} r(s) ds<\frac{1}{2a} \left( 1+\frac{1}{e} 
\right).$
\end{itemize}
Then equation {\rm(\ref{4})} is asymptotically stable.  
 If in addition there exists $T>0$ such that {\rm(\ref{cond1abcd})} holds
then {\rm(\ref{4})} is uniformly exponentially stable. 
\end{corollary}

Next, let us compare Theorems \ref{theorem2} and \ref{theorem3}.

\begin{example}\label{example2}

Consider the equation
\begin{equation}\label{26}
\dot{x}(t)+x(t-1)-0.3x(t)=0,
\end{equation}
which is (\ref{4}) with $a=1,$ $ b=0.3,$ $ r(t)\equiv 1,$ $ h(t)=t-1,$ $ g(t)\equiv t$.
Condition 2)  of Corollary \ref{corollary4} holds, hence  equation (\ref{26}) is uniformly exponentially stable. 
Conditions of Corollary \ref{corollary8} are not satisfied.

Consider the equation
\begin{equation}\label{27}
\dot{x}(t)+0.4x(t-1)-0.35x(t-3)=0
\end{equation}
which is (\ref{5}) with $a=0.4,$ $ b=0.35,$ $ r(t)\equiv 1,$ $ h(t)=t \!- \!1,$ $ g(t)=t \!- \!3.$
Con- \lb
dition 2) of Corollary \ref{corollary7} holds, hence  equation (\ref{27}) is uniformly exponentially stable. 
Conditions of Corollary \ref{corollary3} are not satisfied.

Hence Theorems \ref{theorem2} and \ref{theorem3} are independent.
\end{example}

\begin{example}\label{example2a}
Consider the equation
%a particular case of  (\ref{4}) which generalizes test equation (\ref{3}) 
\begin{equation}\label{3abc}
\dot{x}(t)+\sin^2 t \, (a x(t-2)-bx(t))=0.
\end{equation}
For $a=0.6$ and $b<0.34$, Corollary~\ref{corollary4} implies uniform exponential stability. For this $a$,
Corollary~\ref{corollary8} fails to establish stability of \eqref{3abc}. However, Corollary~\ref{corollary8}
can be applied for 
$$a< \frac{1}{1.4546}\cdot \frac{1}{2} \left( 1+ \frac{1}{e}  \right) \approx 0.47$$
and any $b<a$, and for these values of $a$, Corollary~\ref{corollary4} also implies uniform exponential stability.
\end{example}

\section{Some generalizations}

In this section we consider differential equations with several delays,
integro-differential equations and equations with a distributed delay. 
We will only present generalizations of Theorems~\ref{theorem1}--\ref{theorem3}
to these equations. 
All the corollaries of the generalized theorems can be obtained similarly to the corollaries of 
Theorems~\ref{theorem2} and~\ref{theorem3}.

\subsection{Equations with Several Delays}

Consider an equation with several delays and positive and negative coefficients
\begin{equation}\label{28}
\dot{x}(t)+\sum_{k=1}^m a_k(t)x(h_k(t))-\sum_{k=1}^l b_k(t)x(g_k(t))=0,
\end{equation}
where for the parameters of equation~(\ref{28}) the following conditions hold: 
\begin{itemize}
\item[(b1)] $a_k$ and $b_k$ are Lebesgue measurable essentially bounded functions on \lb
%%\hspace*{-1.7pt}
$[0,\infty)$, $a_k(t)\geq 0$, 
$b_k(t)\geq 0$, 
$ \sum_{k=1}^m a_k(t)\geq \sum_{k=1}^l b_k(t)$;
\item[(b2)] the functions $h_k$ and $g_k$ are Lebesgue measurable  on $[0,\infty)$, and 
$0\leq 
t-h_k(t)\leq \tau$, $0\leq t-g_k(t)\leq \delta$  for some finite constants $\tau$ and  $\delta$.
\end{itemize}
Denote
\begin{equation}
\label{line_1}
%%\begin{alignedat}{2}
a(t) \!= \!\sum_{k=1}^m \! a_k(t),\quad b(t) \!= \!\sum_{k=1}^l \! b_k(t),\quad
 h(t) \!= \!\min_k
h_k(t), \quad H(t) \!= \!\max_k h_k(t),
%%\end{alignedat}
\end{equation}
and
\begin{equation}
\label{line_2}
\begin{alignedat}{2}
g(t)&=\min_k g_k(t),&\quad G(t)&=\max_k g_k(t),\\
 r(t)&=\min\{h(t),g(t)\},&
R(t)&=\max\{H(t),G(t)\}.
\end{alignedat}
\end{equation}
Then  (b1), (b2) imply that  (a1), (a2) hold for $a,b,h,g, H,G$.

%Seite321
\begin{theorem}\label{theorem4}
Assume that {\rm(\ref{cond1})} holds 
and for some $t_0\geq 0$
\[
%%\left.\begin{aligned}
\left\| \!\left(  \int_{h(\cdot)}^{ \, \cdot}  \! \! 
[a(s) \!- \!b(s)] ds \!- \!\frac{1}{e} \!\right)^{ \! \!+}\right\|_{[t_0,\infty)} \! \!
%%&\\
+ 2 \left\| \!\frac{b}{a \!- \!b}\right\|_{[t_0,\infty)} \!\left\|   
\int_{r(\cdot)}^{R(\cdot)} \! \! \! [a(s) \!- \!b(s)] ds
\right\|_{[t_0,\infty)} \! \! \!
%%&\end{aligned}\right\}
< \!1,
\]
where $a,b,h,r,R$ are defined in {\rm(\ref{line_1})} and {\rm(\ref{line_2})}.
Then equation {\rm(\ref{28})} is asymptotically stable. If in addition {\rm(\ref{cond1abc})} holds, {\rm(\ref{28})}
is uniformly exponentially stable.
\end{theorem}
\begin{proof}
Suppose that $x(t)$, $t\geq t_0$ is a solution of the initial value problem
\begin{equation}\label{29}
\begin{alignedat}{2}
\dot{x}(t)+\sum_{k=1}^m a_k(t)x(h_k(t))-\sum_{k=1}^l
b_k(t)x(g_k(t))&=f(t),&\quad
 t&\geq t_0,\\
x(t)&=0,& t&\leq t_0,
\end{alignedat}
\end{equation}
where $f$ is an essentially bounded function on $[t_0,\infty)$. 
By Lemma~\ref{lemma5} there exist the delayed arguments $h_0(t)\leq t$, $h(t)\leq h_0(t)\leq H(t)$, and $g_0(t)\leq t$, $g(t)\leq 
g_0(t)\leq G(t)$ such that
$$
\sum_{k=1}^m a_k(t)x(h_k(t))=a(t)x(h_0(t)), \quad 
 \sum_{k=1}^l b_k(t)x(g_k(t))=b(t)x(g_0(t)),
$$
therefore 
$
\dot{x}(t)+a(t)x(h_0(t))-b(t)x(g_0(t))=f(t).
$
Consider now the delay differential equation
\begin{equation}\label{30}
\dot{y}(t)+a(t)y(h_0(t))-b(t)y(g_0(t))=0.
\end{equation}
We have
\begin{align*}
\left(\int_{h_0(t)}^t [a(s)-b(s)] ds-\frac{1}{e}\right)^{ \! +}&\leq \left(\int_{h(t)}^t [a(s)-b(s)] ds-\frac{1}{e}\right)^{ \! +},\\
%$$$$
\left| \int_{h_0(t)}^{g_0(t)} [a(s)-b(s)] ds \right|&\leq  \int_{r(t)}^{R(t)} [a(s)-b(s)] ds .
\end{align*}
By Theorem~\ref{theorem2}, equation (\ref{30}) is uniformly exponentially stable. 
Hence by Lemma~\ref{lemma1}, the function $x$ which is a solution of a uniformly exponentially stable equation with an 
essentially bounded right-hand side $f$, is 
also an essentially bounded function. 
Thus for any essentially bounded $f$, the solution 
of problem (\ref{29}) is an essentially bounded function. 
Lemma~\ref{lemma2}, Remark~\ref{remark1} and Lemma~\ref{lemma_p}
imply the statement of the theorem.
\qed\end{proof}

\begin{theorem}\label{theorem5}
Assume that  for some $t\geq t_0$ $a(t)\geq a_0>0$, $ a(t)\geq \sum_{k=1}^l b_k(t)$, 
\begin{equation}\label{29b}
\esssup_{t\geq t_0} \left(\frac{1}{a(t)} \sum_{k=1}^l b_k(t) \right) < 1.
\end{equation}
Then the equation 
$$
\dot{x}(t)+a(t)x(t)-\sum_{k=1}^l b_k(t)x(g_k(t))=0
$$
is uniformly exponentially stable.
\end{theorem}
\begin{proof} In this and the next theorem, we follow the scheme of the proof of Theorem~\ref{theorem4}.

Suppose $x(t)$, $t\geq t_0$ is a solution of the initial value problem
\begin{equation*}
%\label{29a}
\begin{alignedat}{2}
\dot{x}(t)+a(t)x(t)-\sum_{k=1}^l b_k(t)x(g_k(t))&=f(t),&\quad t&\geq
t_0,\\
x(t)&=0,& t&\leq t_0,
\end{alignedat}
\end{equation*}
where $f$ is an essentially bounded function on $[t_0,\infty)$. 
By Lemma~\ref{lemma5}, there exists the delayed argument  $g_0(t)\leq t$, $g(t)\leq 
g_0(t)\leq G(t)$ such that
$$
 \sum_{k=1}^l b_k(t)x(g_k(t))=b(t)x(g_0(t)),
$$
therefore 
$
\dot{x}(t)+a(t)x(t)-b(t)x(g_0(t))=f(t).
$
Consider now the delay differential equation
\begin{equation}\label{30a}
\dot{y}(t)+a(t)y(t)-b(t)y(g_0(t))=0.
\end{equation}
Inequality (\ref{29b}) and Theorem~\ref{theorem1} imply that equation (\ref{30a}) is uniformly exponentially stable. 
The proof is concluded similarly to the end of the proof of Theorem~\ref{theorem4}.
\qed\end{proof}

\pb

\begin{theorem}\label{theorem6}
Assume that condition {\rm(\ref{cond2})} holds, for some $t_0\geq 0$ 
\begin{equation}\label{30b}
\esssup_{t\geq t_0}\frac{b(t)}{a(t)}<1,\quad \esssup_{t\geq t_0}\int_{r(t)}^{R(t)} a(s) ds<1,
\end{equation}
and for some $t_0\geq 0$ we have
\begin{equation}\label{30c}
\left\|\left(  \int_{h(\cdot)}^{ \, \cdot} a(s) ds-\frac{1}{e}\right)^{ \! +}\right\|_{[t_0,\infty)}
\hspace*{-10pt}<\frac{1-\left\|\frac{b}{a}\right\|_{[t_0,\infty)}}{\left\|1-\frac{b}{a}\right\|_{[t_0,\infty)}}
\left(1-\left\|  \int_{r(\cdot)}^{R(\cdot)} a(s) ds\right\|_{[t_0,\infty)}\right),
\end{equation}
where $a,b,h,r,R$ are defined in {\rm(\ref{line_1})} and {\rm(\ref{line_2})}.
Then equation {\rm(\ref{28})} is asymptotically stable. If in addition {\rm(\ref{18b})} holds,  equation {\rm(\ref{28})}
is uniformly exponentially stable.
\end{theorem}
\begin{proof}
Suppose $x(t)$, $t\geq t_0$ is a solution of the initial value problem (\ref{29}),
where $f$ is an essentially bounded function on $[t_0,\infty)$. 
By Lemma~\ref{lemma5}, there exist the delayed arguments $h_0(t)\leq t$, $h(t)\leq h_0(t)\leq H(t)$ and $g_0(t)\leq t$, $g(t)\leq 
g_0(t)\leq G(t)$ such that
$$
\sum_{k=1}^m a_k(t)x(h_k(t))=a(t)x(h_0(t)), \quad 
 \sum_{k=1}^l b_k(t)x(g_k(t))=b(t)x(g_0(t)),
$$
therefore 
$
\dot{x}(t)+a(t)x(h_0(t))-b(t)x(g_0(t))=f(t).
$
Consider now  equation (\ref{30}).
We have
\begin{align*}
\left(\int_{h_0(t)}^t a(s) ds-\frac{1}{e}\right)^{ \! +}&\leq \left(\int_{h(t)}^t a(s) ds-\frac{1}{e}\right)^{ \! +},\\
%$$$$
\left| \int_{h_0(t)}^{g_0(t)} a(s) ds \right|&\leq  \int_{r(t)}^{R(t)} a(s) ds.
\end{align*}
By (\ref{30b}), (\ref{30c}) and Theorem~\ref{theorem3}, equation (\ref{30}) is uniformly exponentially stable.
The rest of the proof is the same as in the proof of Theorem~\ref{theorem4}.
\qed\end{proof}

\subsection{Equations with Distributed Delays}

Consider the equation with distributed delays
\begin{equation}\label{33}
\dot{x}(t)+ a(t) \int_{h(t)}^t x(s)d_s A(t,s) -b(t)\int_{g(t)}^t x(s) d_s B(t,s)=0,
\end{equation}
where $a$, $b$, $h$, $g$ satisfy (a1), (a2), $A(t,s), B(t,s)$ are measurable on $[0,\infty)\times [0,\infty)$,
$A(t, \cdot)$ and $B(t, \cdot)$ are left continuous non-decreasing functions
for almost all  $t$, $A(\cdot,s)$ and $B(\cdot,s)$  are locally integrable for
any $s$, $A(t,h(t)) \!= \!B(t,g(t)) \!= \!0,$ and $A(t,t^+)=B(t,t^+)=1$.
Then 
$$
\int_{h(t)}^t d_s A(t,s)=\int_{g(t)}^t d_s B(t,s)=1.
$$
Denote 
\begin{equation*}
%\label{33a}
u(t)=\min\{h(t),g(t)\}, \quad U(t)=\max\{h(t),g(t)\}.
\end{equation*}

\begin{theorem}
\label{theorem10}
Assume that condition  {\rm(\ref{cond1})} holds 
and for some $t_0\geq 0$
\[
%%\left.\begin{aligned}
\left\| \!\left(  \int_{h(\cdot)}^{ \, \cdot}  \! \! 
[a(s) \!- \!b(s)] ds \!- \!\frac{1}{e} \!\right)^{ \! \!+}\right\|_{[t_0,\infty)} \! \!
%%&\\
+ 2 \left\| \!\frac{b}{a \!- \!b}\right\|_{[t_0,\infty)} \!\left\|   
\int_{u(\cdot)}^{U(\cdot)} \! \! \! [a(s) \!- \!b(s)] ds
\right\|_{[t_0,\infty)} \! \! \!
%%&\end{aligned}\right\}
< \!1,
\]
Then equation {\rm(\ref{33})} is asymptotically stable. If in addition {\rm(\ref{cond1abc})} holds, 
{\rm(\ref{33})} is uniformly exponentially stable.
\end{theorem}
\begin{proof}
Suppose that for $t\geq t_0$, $x$ is a solution of the  initial value problem
\begin{equation}\label{32}
\begin{alignedat}{2}
\dot{x}(t)+a(t)\int_{h(t)}^t  x(s)d_sA(t,s)-b(t)\int_{g(t)}^t x(s) d_s
B(t,s)&=f(t),&\quad t&\geq t_0,\\
x(t)&=0,& \quad t&\leq t_0,
\end{alignedat}
\end{equation}
where $f$ is an essentially bounded function on $[t_0,\infty)$. 
By Lemma~\ref{lemma7}  there exist functions $h_0(t)\leq t$,  $h(t)\leq h_0(t)\leq t$ 
and $g_0(t)\leq t$, $g(t)\leq g_0(t)\leq t$
such that
$$
\int_{h(t)}^t  x(s)d_s A(t,s)= x(h_0(t)),\quad
 \int_{g(t)}^t  x(s)d_s B(t,s)= x(g_0(t)),
$$
hence $x$ satisfies the equation
\begin{equation}
\label{43add}
\dot{x}(t)+a(t)x(h_0(t))-b(t)x(g_0(t))=f(t).
\end{equation}
By Theorem~\ref{theorem2}, equation (\ref{43add}) is uniformly exponentially stable.
Lemma~\ref{lemma1} yields that the solution $x$ of a uniformly exponentially stable equation with an essentially 
bounded right-hand side $f$ is an essentially bounded function. 
Thus, for any essentially bounded function $f$, the solution
of problem (\ref{32}) is essentially bounded.
By Lemma~\ref{lemma2}, Remark~\ref{remark1} and Lemma~\ref{lemma_p}, equation  (\ref{33}) is asymptotically stable.
Also, by Lemma~\ref{lemma_p}, under (\ref{cond1abc}) 
equation  (\ref{33}) is
uniformly exponentially stable.
\qed\end{proof}

The proofs of the following two theorems are similar to the proofs 
of Theorems~\ref{theorem5} and~\ref{theorem6}  and thus are omitted.

\begin{theorem}\label{theorem11}
Assume that $a(t)\geq a_0>0$ and for some $t_0\geq 0$
$$
\esssup_{t\geq t_0} \frac{b(t)}{a(t)}<1.$$
Then the equation 
\begin{equation*}
%\label{33b}
\dot{x}(t)+a(t)x(t)- b(t) \int_{g(t)}^t x(s) d_s B(t,s)=0
\end{equation*}
is asymptotically stable, and, if in addition {\rm(\ref{cond1abc})} holds,
%equation (\ref{33b}) is 
uniformly exponentially stable.
\end{theorem}

\begin{theorem}\label{theorem12}
Assume that condition {\rm(\ref{cond2})} holds,  for some $t_0\geq 0$ 
$$
\esssup_{t\geq t_0}\frac{b(t)}{a(t)}<1,\quad \esssup_{t\geq t_0} \int_{u(t)}^{U(t)} a(s) ds<1,
$$
 
$$
\left\|\left(  \int_{h(\cdot)}^{ \, \cdot} a(s) ds-\frac{1}{e}\right)^{ \! +}\right\|_{[t_0,\infty)}
<\frac{1-\left\|\frac{b}{a}\right\|_{[t_0,\infty)}}{\left\|1-\frac{b}{a}\right\|_{[t_0,\infty)}}
\left(1-\left\|  \int_{u(\cdot)}^{U(\cdot)} a(s) ds\right\|_{[t_0,\infty)}\right).
$$
Then equation {\rm(\ref{33})} is asymptotically stable and, if in addition {\rm(\ref{18b})} holds,
%(\ref{33}) is 
uniformly exponentially stable.
\end{theorem}

\subsection{Integro-differential Equations}

The integro-differential equation
\begin{equation}\label{31}
\dot{x}(t)+\int_{h(t)}^t K(t,s) x(s)ds-\int_{g(t)}^t P(t,s) x(s)ds=0,
\end{equation}
where $K(t,s)$ and $P(t,s)$ are Lebesgue 
measurable locally integrable on $[0,\infty)\times [0,\infty)$ functions, $K(t,s)\geq 0, 
P(t,s)\geq 0$, is a particular case of (\ref{33}).
After denoting 
$$
a(t)=\int_{h(t)}^t K(t,s)ds, \quad b(t)=\int_{g(t)}^t P(t,s)ds,
$$
\begin{align*}
A(t,s) &= \left\{ \begin{array}{ll}   \frac{1}{a(t)} \int_{h(t)}^s K(t,\zeta)~d\zeta, \ &  a(t)>0,
\\ 0, & a(t)=0,  \end{array} \right. \\
B(t,s) &= \left\{ \begin{array}{ll}   \frac{1}{b(t)} \int_{g(t)}^s P(t,\zeta)~d\zeta, \ \ &  b(t)>0,
\\ 0, & b(t)=0,  \end{array} \right. 
\end{align*}
equation (\ref{31}) has the form of (\ref{33}).

Assume that for the functions $a,b,h,g$ conditions (a1), (a2) hold.
Denote $u(t)=\min\{h(t),g(t)\},$ $U(t)=\max\{h(t),g(t)\}$.
The following theorems are corollaries of Theorems~\ref{theorem10}--\ref{theorem12}.

\begin{theorem}\label{theorem7} %Seite326
Assume that condition  {\rm(\ref{cond1})} holds 
and for some $t_0\geq 0$
\[
%%\left.\begin{aligned}
\left\| \!\left(  \int_{h(\cdot)}^{ \, \cdot}  \! \! 
[a(s) \!- \!b(s)] ds \!- \!\frac{1}{e} \!\right)^{ \! \!+}\right\|_{[t_0,\infty)} \! \!
%%&\\
+ 2 \left\| \!\frac{b}{a \!- \!b}\right\|_{[t_0,\infty)} \!\left\|   
\int_{u(\cdot)}^{U(\cdot)} \! \! \! [a(s) \!- \!b(s)] ds
\right\|_{[t_0,\infty)} \! \! \!
%%&\end{aligned}\right\}
< \!1,
\]
Then equation 
{\rm(\ref{31})} is asymptotically stable and, if in addition {\rm(\ref{cond1abc})} holds,
%(\ref{31}) is 
uniformly exponentially stable.
\end{theorem}

\begin{theorem}\label{theorem8}
Assume that $a(t)\geq a_0>0$, for some $t_0\geq 0$
$ 
\esssup_{t\geq t_0} \frac{b(t)}{a(t)}<1$.
Then the equation 
$$
\dot{x}(t)+a(t)x(t)-\int_{g(t)}^t P(t,s) x(s)ds=0
$$
is asymptotically stable and, if in addition {\rm(\ref{cond1abc})} holds, 
%it is 
uniformly exponentially stable.
\end{theorem}

\begin{theorem}\label{theorem9}
Assume that condition {\rm(\ref{cond2})} holds, and for some $t_0\geq 0$
$$
\esssup_{t\geq t_0} \frac{b(t)}{a(t)}<1,\quad \esssup_{t\geq t_0}  \int_{u(t)}^{U(t)} a(s) ds<1,
$$
$$
\left\|\left(  \int_{h(\cdot)}^{ \, \cdot} a(s) ds-\frac{1}{e}\right)^{ \! +}\right\|_{[t_0,\infty)}
<\frac{1-\left\|\frac{b}{a}\right\|_{[t_0,\infty)}}{\left\|1-\frac{b}{a}\right\|_{[t_0,\infty)}}
\left(1-\left\|  \int_{u(\cdot)}^{U(\cdot)} a(s) ds\right\|_{[t_0,\infty)}\right).
$$
%where $\| \cdot \|$ is the sup-norm on $[t_0, \infty)$.
Then equation {\rm(\ref{31})} is asymptotically stable and,
if in addition {\rm(\ref{18b})} holds, 
%(\ref{31}) is 
uniformly exponentially stable.
\end{theorem}

\section{Mackey--Glass equations}

We recall that an equilibrium of a nonlinear delay differential equation is {\em locally 
exponentially stable} (LES) if the linearized equation is uniformly exponentially stable.
In this section we consider Mackey--Glass equation \eqref{51}
under the following conditions which are assumed without further mentioning them:
\begin{itemize}
\item[(c1)] $r$ is a measurable  essentially bounded function on $[t_0,\infty)$ function, 
$r(t)\geq 0$, $r(t)\neq 0$ a.e.,  and for some $T>0$ (\ref{18b}) holds;
%$ \int_{t_0}^{\infty} r(s)ds=\infty$;
\item[(c2)] $h,g$ are measurable functions, and there exist $\tau>0$ and $\sigma>0$ such that  $0\leq t-h(t)\leq\tau$ and $0\leq 
t-g(t)\leq\sigma$;
\item[(c3)] $n, \beta, \gamma$ are positive constants. 
\end{itemize}
Equation \eqref{MG} is considered satisfying (c1), (c2) with $r,h,g$ replaced by  $s, p,q$ and $\beta>1$, $n>0$.

We will obtain LES conditions for the positive equilibrium. 
LES conditions for the trivial equilibrium can be obtained similarly.

\begin{theorem}\label{theorem13}
Assume that 
\begin{equation}\label{55a}
\beta>\gamma, \;n\leq 1\quad\mbox{or}\quad n>1, \;1<\frac{\beta}{\gamma}<1+\frac{1}{n-1}
\end{equation}
and at least one of the conditions of either Corollary~{\rm\ref{corollary3}} or {\rm\ref{corollary7}} holds, where
$a=\gamma$, $b=\gamma\big(1-n+\frac{\gamma n}{\beta}\big)$.
Then the positive  equilibrium  of equation {\rm(\ref{51})} is~LES. 
\end{theorem}

\begin{proof}
Linearized equation for  (\ref{51}) around the positive equilibrium  has  form (\ref{52}). Condition (\ref{55a})
implies that in (\ref{52}) $a>b$.
 Corollaries ~\ref{corollary3} and \ref{corollary7} imply that equation (\ref{52}) 
is uniformly exponentially stable. Then the positive  equilibrium  of equation (\ref{51}) is LES. 
\qed\end{proof}

\begin{remark}\label{remark5}
If $1-n+\frac{\gamma n}{\beta}<0$ then  (\ref{52}) is an equation with two positive coefficients. Explicit
exponential stability conditions for such equations with measurable parameters can be found in \cite{BB1, BB2, GD}.
Global stability of (\ref{51}) with $g(t)\equiv t$ was studied in  \cite{BBI2013}.
\end{remark}

Next, we present LES conditions for  the positive equilibrium $x^{\ast}=$ \lb
$(\beta-1)^{\frac{1}{n}}$ of equation (\ref{MG}).
As (\ref{lin}) involves more than two terms, we will apply Theorems~\ref{theorem4} and \ref{theorem6} 
for equations  with several delays.

\begin{theorem}\label{theorem15}
Assume that for equation {\rm(\ref{MG})} either
\begin{equation}\label{56}
\left\|\left(\alpha\int_{q(\cdot)}^{ \, \cdot} s(\tau) d\tau-\frac{1}{e}\right)^{ \! +}\right\|_{[t_0,\infty)}
+2\left\|\int_{\min\{p(\cdot),q(\cdot)\}}^{\max\{p(\cdot),q(\cdot)\}}s(\tau)d\tau\right\|_{[t_0,\infty)}<1
\end{equation}
is satisfied, or both inequalities below hold: 
\begin{equation}\label{57}
(1+\alpha)\left\|\int_{\min\{p(\cdot),q(\cdot)\}}^{\max\{p(\cdot),q(\cdot)\}}s(\tau)d\tau\right\|_{[t_0,\infty)}<1,
\end{equation}
\begin{equation}\label{58}
\left\|\left((1+\alpha)\int_{q(\cdot)}^{ \, \cdot} s(\tau) d\tau-\frac{1}{e}\right)^{ \! +}\right\|_{[t_0,\infty)}\hspace*{-21pt}<
1-\left\|\int_{\min\{p(\cdot),q(\cdot)\}}^{\max\{p(\cdot),q(\cdot)\}} (1+\alpha) s(\tau)d\tau\right\|_{[t_0,\infty)}
\end{equation}
where $ \alpha=\frac{n(\beta-1)}{\beta}$. Then the positive equilibrium of equation {\rm(\ref{MG})} is LES.
\end{theorem}
\begin{proof}
The linearized equation for  (\ref{MG}) about the equilibrium has  form (\ref{lin}).
%\begin{equation}\label{59}
%\dot{x}(t)+s(t)[x(t)+\alpha x(p(t))-x(q(t))]=0.
%\end{equation}
In order to apply Theorem~\ref{theorem4} to equation (\ref{lin}), denote
$$
a(t)=(1+\alpha)s(t), \quad b(t)=s(t), \quad h(t)=p(t), \quad g(t)=q(t),$$ $$
r(t)=\min\{p(t),q(t)\}, \quad R(t)=\max\{p(t),q(t)\}.
$$
If (\ref{56}) holds then conditions of Theorem~\ref{theorem4} are satisfied for (\ref{lin}), and so
the positive equilibrium of equation (\ref{MG}) is LES.

By Theorem~\ref{theorem6}, inequalities (\ref{57}) and (\ref{58}) 
imply uniform exponential stability for the zero solution of (\ref{lin}), thus  the positive equilibrium of equation (\ref{MG}) is LES.
\qed\end{proof}

\begin{figure}[ht]
\centering
\mycomment{
\vspace*{-55mm}
\hspace*{-2mm}
\includegraphics[scale=.35]{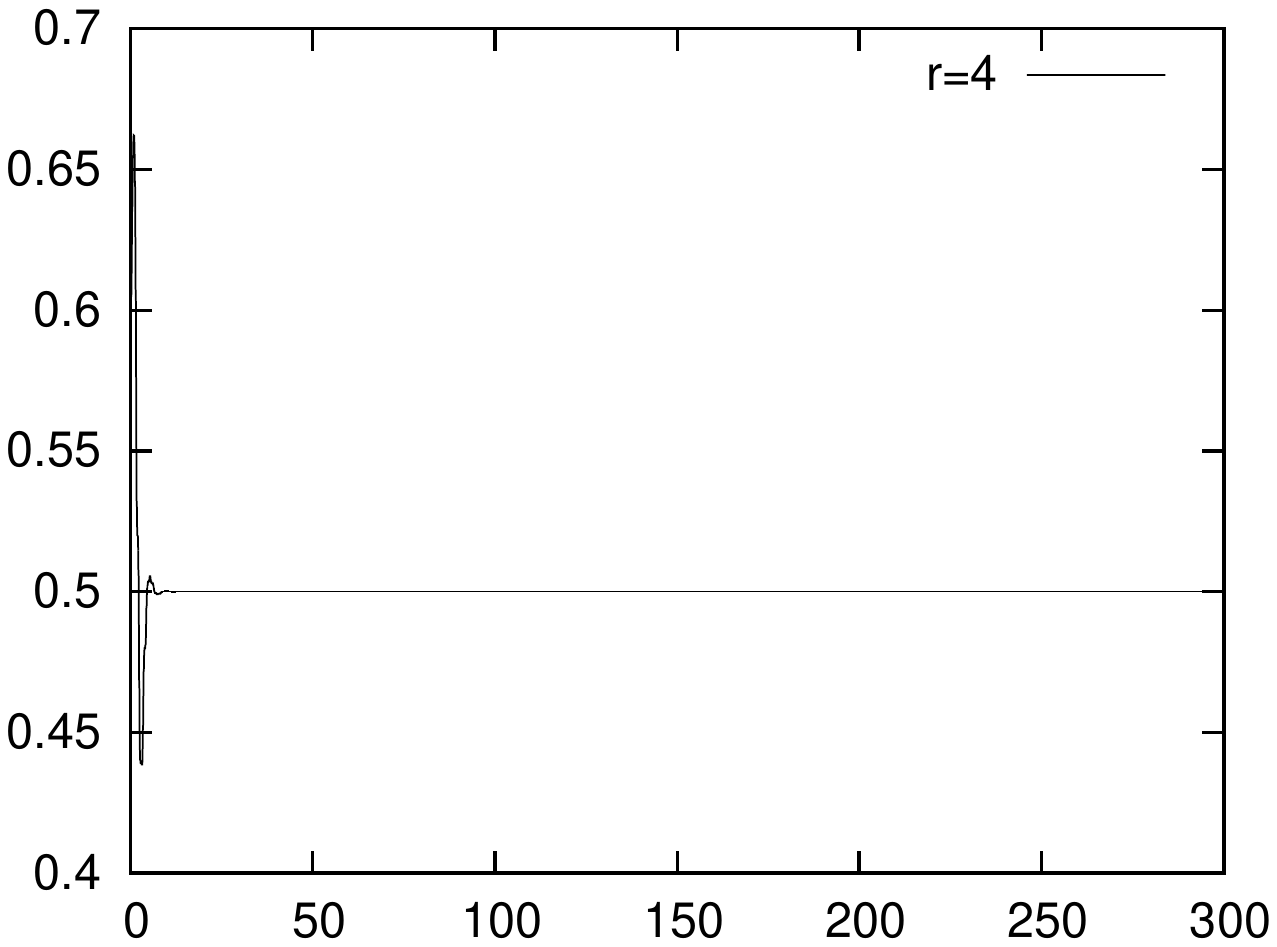}
\hspace{-15mm}
\includegraphics[scale=.35]{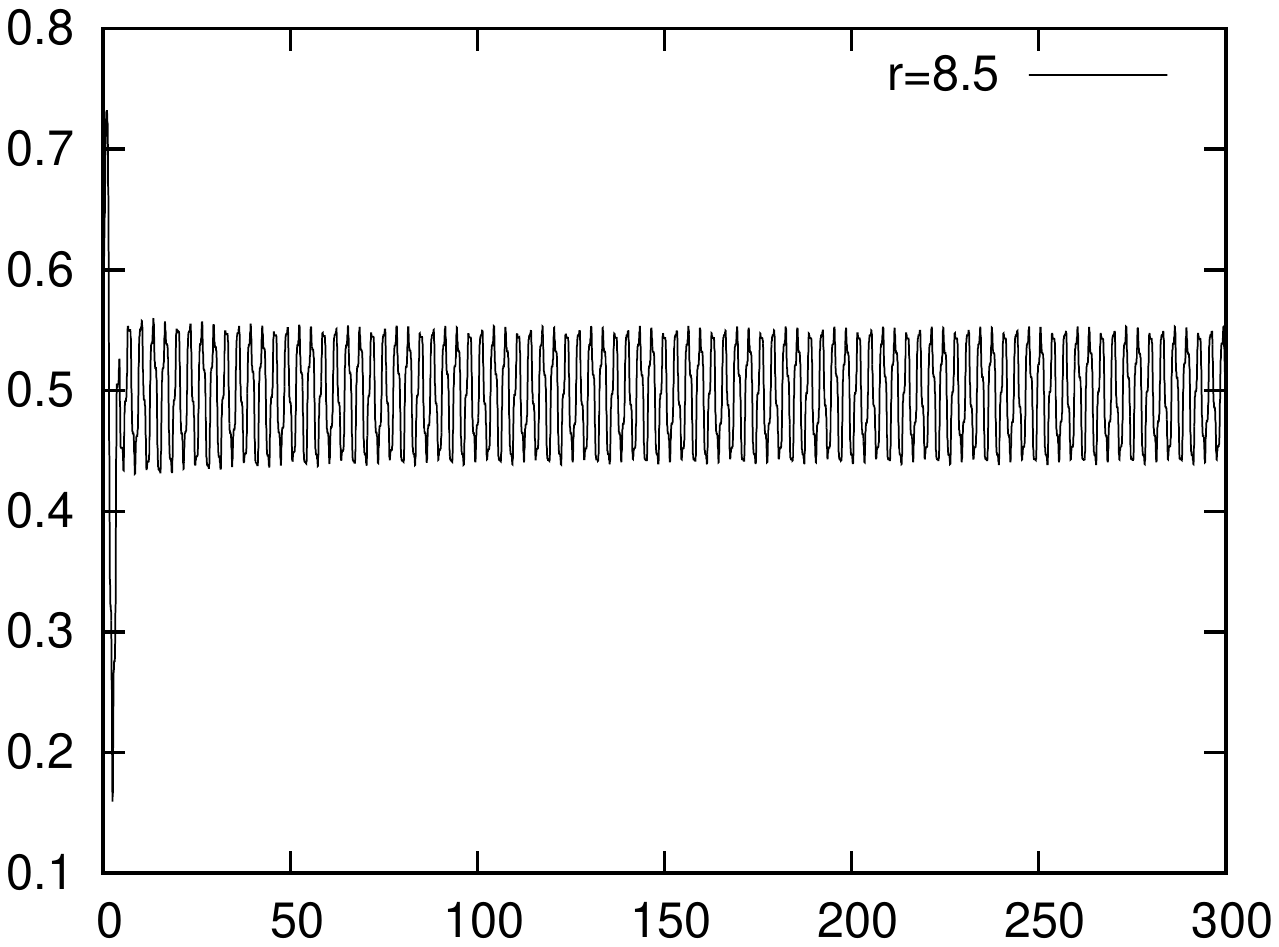}}
\vspace*{-12mm}
\hspace*{2mm}
\caption{A solution of  equation (\protect{\ref{ex_51}}) with $\beta=1.25$, $\tau=1$, $x^{\ast}=0.5$, $x(0)=0.6$,
$\varphi(t) \equiv 0.4$, $t<0$ and (left) $\sigma=1.1$, $r=4$, (right) $\sigma=1.1$, $r=8.5$.
}
\label{figure1}
\end{figure}

\begin{example}
Consider a particular case of (\ref{51})
\begin{equation}\label{ex_51}
\dot{x}(t)= r \sin^2 (\pi t)  \left[\beta \frac{x(t-\sigma)}{1+x^2(t-\sigma)}-  x(t-1)\right],
\end{equation}
with $\beta >1$, its linearized about $x^{\ast}=\sqrt{\beta-1}$ version is 
%\begin{equation}\label{ex_52}
\[
\dot{x}(t)+ r \sin^2 (\pi t) \left[ x(t-1) -\left(\frac{2}{\beta} - 1 \right) x(t-\sigma) \right] =0.
\]
%\end{equation}
Let $\beta=1.25$. For (\ref{ex_51}), LES is guaranteed if any of the conditions of Corollaries~\ref{corollary3} or \ref{corollary7}
are satisfied for $a=1$, $b=0.6$, $h(t)=t-1$, $g(t)=t-\sigma$. Here $a-b=0.4$, $r(t)=r \sin^2 (\pi t)$,
$$
\esssup_{t\geq t_0} \int_{h(t)}^t r(s)ds =\frac{r}{2}.
$$
For $\sigma=1.1$, 
$$
\left| \int_{h(t)}^{g(t)} r(s)ds \right| \leq r \left( 0.05  +\frac{\sin(0.1)}{2} \right) \approx 0.0999917r<0.1r,
$$
as $|\cos(t)|\leq 1$.
Condition 1) in Corollary~\ref{corollary3}
is satisfied for $r\leq \frac{5}{e} \approx 1.8393972$,
while 2) holds if $r> \frac{5}{e}$ and $(0.32)r<1+\frac{1}{e}$, or $r \leq r_1 \approx 4.27$.
Note that Corollary~\ref{corollary7} here gives a worse estimate, as it involves  $r<2$. 

In Figure~\ref{figure1}, left, $r=4<4.27$, thus the conditions of Corollary~\ref{corollary3}
are satisfied, and we observe stability as predicted. For $r=8.5>4.27$ in Figure~\ref{figure1}, right,
the conditions are not satisfied, and we see sustainable oscillations. 

%set size 1.2,0.6
%set output "r_4_b_1.25_tau_1_sigma_1.1.ps"
%set output "r_8.5_b_1.25_tau_1_sigma_1.1.ps"

%set term postscript portrait
%set size 0.8,0.4
%set output "r_4_b_125_tau_1_sigma_11_a.ps"
%plot [0:300] "r_4_b_1.25_tau_1_sigma_1.1" w l
%set output "r_85_b_125_tau_1_sigma_11_a.ps"
%plot [0:300] "r_8.5_b_1.25_tau_1_sigma_1.1" w l

\begin{figure}[ht]
\centering
%\includegraphics[height=.18\textheight]{r_6_b_125_tau_1_sigma_1.ps}
%\hspace{4mm}
%\includegraphics[height=.18\textheight]{r_8_b_125_tau_1_sigma_1.ps}
\mycomment{
\vspace{-55mm}
\hspace*{-2mm}
\includegraphics[scale=.35]{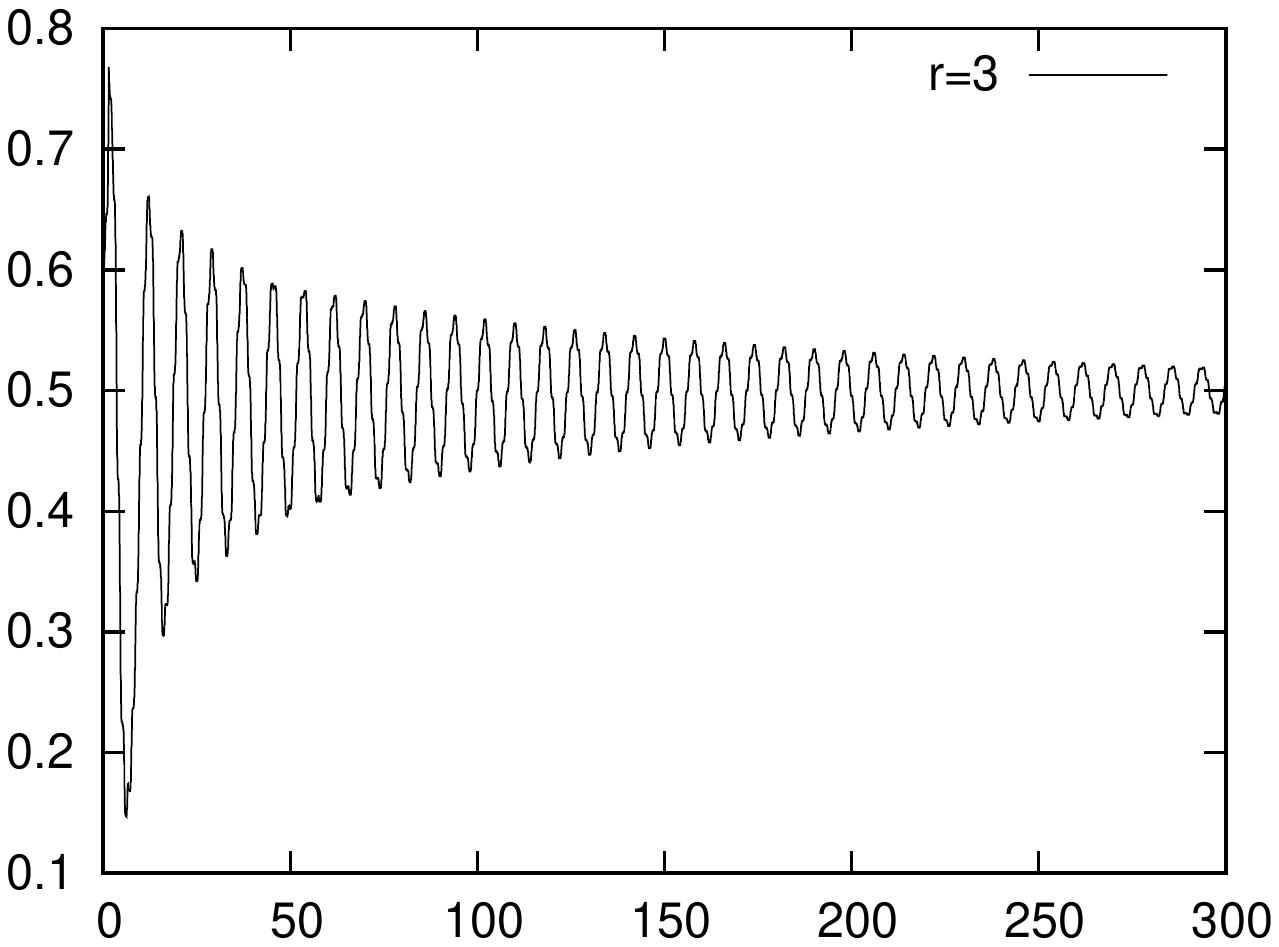}
\hspace{-15mm}
\includegraphics[scale=.35]{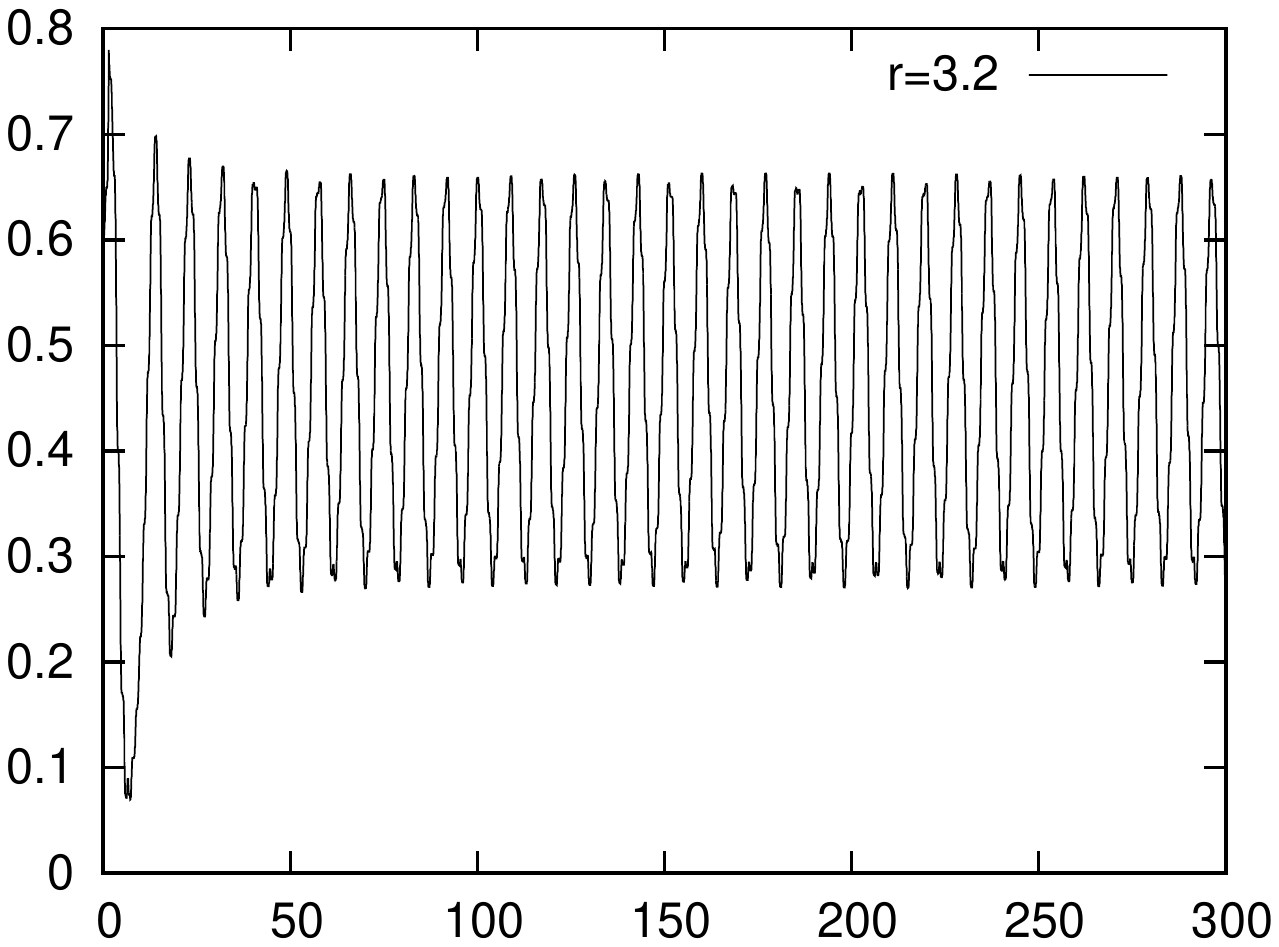}}
\vspace*{-12mm}
\hspace*{2mm}
\caption{A solution of  equation (\protect{\ref{ex_51}}) with $\beta=1.25$, $\tau=1$, $x^{\ast}=0.5$, $x(0)=0.6$,
$\varphi(t) \equiv 0.4$, $t<0$ and (left) $\sigma=1.5$, $r=3$, (right) $\sigma=1.5$, $r=3.2$.
}
\label{figure1a}
\end{figure}

For $\sigma=1.5$ we have
$$
\limsup_{t\rightarrow\infty} \left| \int_{h(t)}^{g(t)} r(s)ds \right| =\frac{r}{4}.
$$
The conditions of Corollary~\ref{corollary3} hold for $r<2+\frac{2}{e}\approx 2.73576$ and, similarly,
Corollary~\ref{corollary7} gives a worse estimate. However, here we illustrate that for 
$r=3$ where the conditions of Corollary~\ref{corollary3} are no longer satisfied, in Figure~\ref{figure1a}, left,
we still observe stability. However, for $r=3.2$ (Figure~\ref{figure1a}, right) the positive equilibrium is no longer 
stable. Let us note that the predicted and obtained in simulations instability bounds are not very far from each other.

%set term postscript portrait
%set size 1.2,0.6
%set output "r_6_b_1.25_tau_1_sigma_1.ps"
%plot [0:300] "r_6_b_1.25_tau_1_sigma_1" w l
%set output "r_8_b_1.25_tau_1_sigma_1.ps"
%plot [0:300] "r_8_b_1.25_tau_1_sigma_1" w l

%%set output "r_3.2_b_1.25_tau_1_sigma_1.5.ps"
%set output "r_32_b_125_tau_1_sigma_15_a.ps"
%plot [0:300] "r_3.2_b_1.25_tau_1_sigma_1.5" w l
%%set output "r_3_b_1.25_tau_1_sigma_1.5.ps"
%set output "r_3_b_125_tau_1_sigma_15_a.ps"
%plot [0:300] "r_3_b_1.25_tau_1_sigma_1.5" w l

\end{example}

Next, let us compare LES conditions for equation (\ref{MG}) with  known local stability tests.

\begin{example}\label{example5}
Consider the equation
\begin{equation}
\label{ex_5}
\dot{x}(t) =0.1\sin^2 (\pi t) \left[ \frac{2x(t-3)}{1+x^{n}(t-6)} - x(t) \right], \quad t \geq 0,
\end{equation}
Here $\alpha=\frac{n}{2}$.
By the first part of Theorem~\ref{theorem15} equation (\ref{ex_5}) is LES if 
$$\left( 0.1 \cdot \frac{n}{2} \cdot 3 - \frac{1}{e} \right)^{ \! +} +2 \cdot \frac{3}{2} \cdot 0.1 <1,$$
which is satisfied for $n <7.119$. The second part of Theorem~\ref{theorem15} requires
$$ 0.1 \cdot \frac{3}{2} \left( 1+\frac{n}{2} \right) <1,$$
which holds for $n<11.333$, and in addition
$$ \left(0.1 \cdot \frac{3}{2} \left( 1+\frac{n}{2} \right)
- \frac{1}{e}\right)^{ \! +} +0.1 \cdot \frac{3}{2}<1,$$
which is satisfied for $n<14.2$. Overall, Theorem~\ref{theorem15} implies LES of the equilibrium $x^{\ast}=1$ of 
(\ref{MG}) for $n<11.333$.  
Figure~\ref{figure2} illustrates solutions for $n=11$ where we observe stability of $x^{\ast}$ and $n=13$ 
where sustainable oscillations about $x^{\ast}$ are observed.

The assumptions of \cite[Theorem 3.5]{JMAA2017} are satisfied if $n<\frac{8}{3}$.
Further, conditions (a) and (b) of  \cite[Theorem 3.8]{JMAA2017} imply $n<4$, while (c) includes
$0.6(2+\frac{n}{2})<1$, which cannot be satisfied, while (d) implies $n<\frac{7}{3}$. We see that the results of 
Theorem~\ref{theorem15} for $n \in [4,11]$,
establishes LES of (\ref{ex_5}) while the tests of \cite{JMAA2017} fail. We are not aware of other stability 
conditions that can be applied to (\ref{ex_5}).

\begin{figure}[ht]
\centering
\mycomment{
\vspace*{-55mm}
\hspace*{-2mm}
%\vspace*{-30mm}
%\hspace*{-6mm}
\includegraphics[scale=.35]{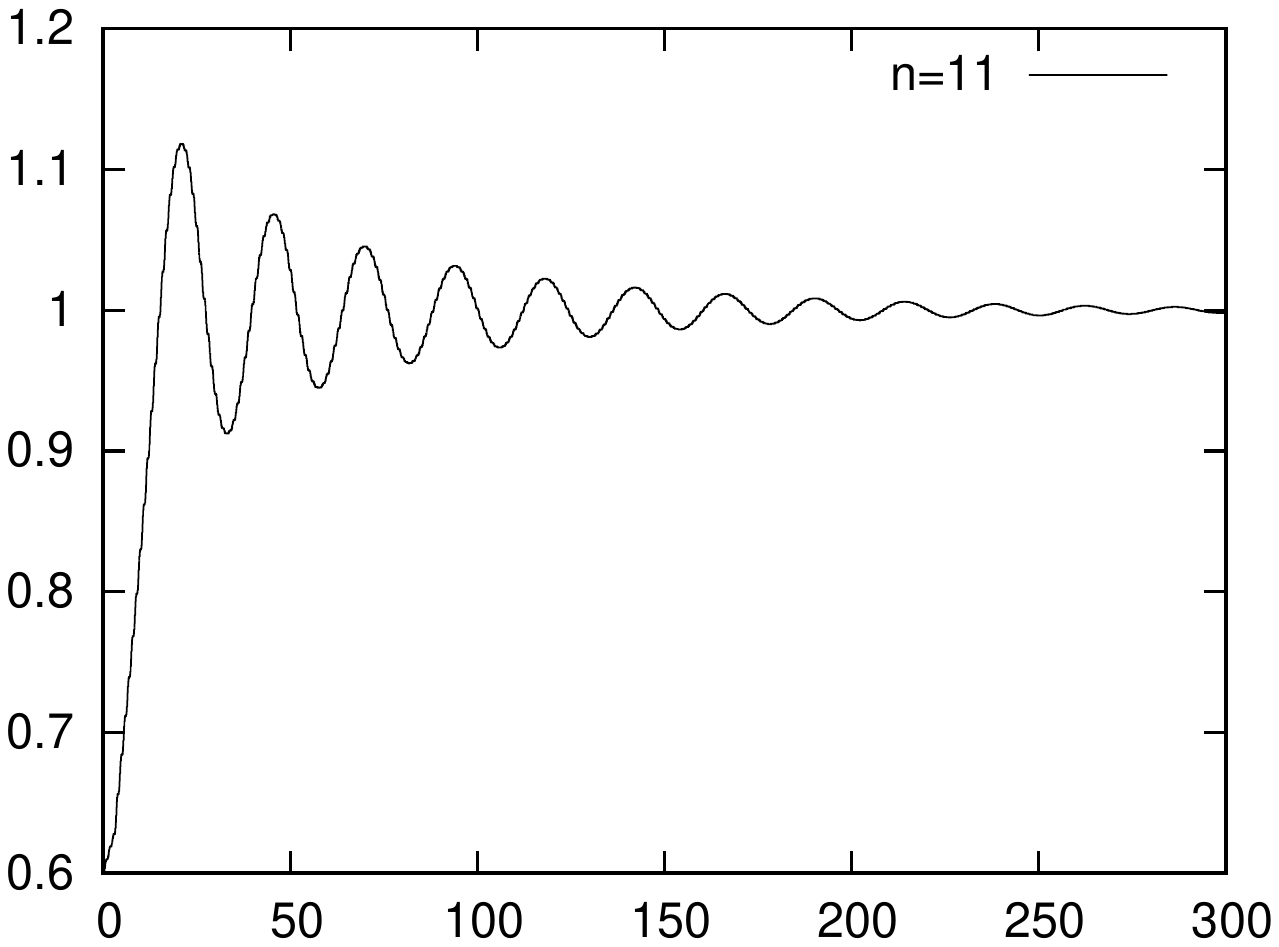}
\hspace*{-15mm}
\includegraphics[scale=.35]{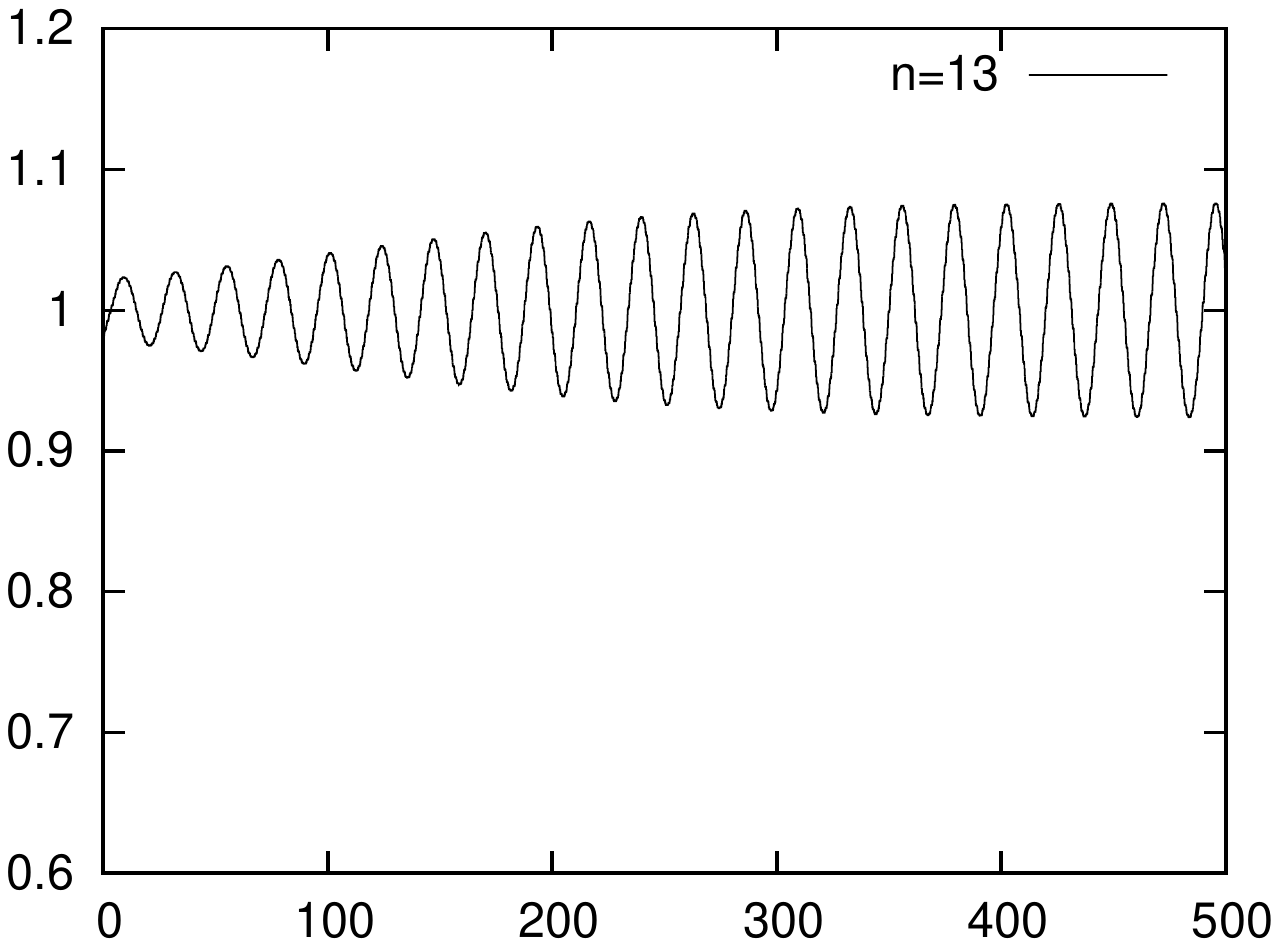}}
\vspace*{-12mm}
\hspace*{2mm}
\caption{A solution of  equation (\protect{\ref{ex_5}}) with (left)  $x(0)=0.6$,
$\varphi(t)=0.4$, $t<0$, $n=11$ and (right) $x(0)=0.98$,$\varphi(t)=0.98$, $t<0$ $n=13$.
}
\label{figure2}
\end{figure}

%set term postscript portrait
%set size 1.2,0.6
%set size 0.8,0.4
%set output "K_1_r_11.ps"
%set output "K_1_r_11_a.ps"
%plot [0:300] "K_1_r_11" w l
%set output "K_1_r_13.ps"
%set output "K_1_r_13_a.ps"
%plot [0:500][0.6:1.2] "K_1_r_13" w l
\end{example} 

\section{Discussion and Open Problems}

Investigation of local exponential stability for Mackey--Glass type models considered 
in the paper, or for nonlinear models with harvesting, leads to delay differential equations
with positive and negative coefficients (with or without a non-delay term) as linearized equations. 
However, even if such a non-delay term exists, it usually does not
dominate over the other terms. For such equations only few explicit stability conditions are known.
The present paper fills the gap.

All the equations considered are in the most general setting: the parameters are measurable,
and solutions are absolutely continuous functions. We obtain an exponential stability condition 
for test equation (\ref{5}), which is sharper than other known stability results. 
However, for some equations known stability results can be better, 
for example, stability results obtained for autonomous equations by the direct investigation of the roots of quasi-polynomials 
should outperform general results applied to this class of equations.

In the present paper we obtained local exponential stability results for Mackey--Glass type equation (\ref{51}).
Similarly, we can consider
%Using stability results of the paper one can study stability problems of other mathematical models with 
%several delays, for example,  Mackey--Glass type equation (\ref{MG}), 
an integro-diffe\-rential Mackey--Glass type equation
$$
\dot{x}(t)=r(t)\left(\beta\int_{h(t)}^t K(t,s) x(s)ds-\gamma x(g(t))\right)
$$
or a Mackey--Glass type equation with a distributed delay
$$
\dot{x}(t)=r(t)\left(\beta\int_{h(t)}^t  x(s)d_s R(t,s) -\gamma x(g(t))\right).
$$

One of open problems is to obtain explicit {\em instability} results for both linear and Mackey--Glass
equations considered in the paper. Some other open problems and topics for future research are listed below.
%For nonlinear equations any stability result for the linearized equation about an equilibrium
%implies LAS of the appropriate equilibrium.

\begin{enumerate}
\item
It would be interesting to obtain new exponential stability conditions for (\ref{1}), which would allow to deduce the 
sharp exponential stability result for test equation (\ref{3}).

\item
Suppose that for equation (\ref{1}) conditions (a1), (a2) hold, $a(t)-b(t)\geq a_0>0$ and this equation is non-oscillatory.
Prove or disprove that (\ref{1}) is uniformly exponentially stable.

\item
Consider (\ref{1}) with $a(t)-b(t)$ being an oscillatory function.
The equation
$$
\dot{x}(t)+\sum_{k=1}^m a_k(t)x(h_k(t))=0,
$$
where all coefficients $a_k$ or part of them are eventually oscillatory functions,
was recently investigated in the papers \cite{BB5,BB4}.
Extend the results of \cite{BB5,BB4} to the equations
%Similarly, consider the equation
$$
\dot{x}(t)+\int_{h(t)}^t x(s) K(t,s)ds=0
$$
with an oscillatory kernel $K(t,s)$ and
$$
\dot{x}(t)+\int_{h(t)}^t x(s) d_s R(t,s)=0,
$$
where $R(t,s)$ can be both increasing and decreasing in the second argument.

\item
Consider all equations in the paper without the assumption that delay functions are bounded in condition (a2).
Is it possible to deduce asymptotic stability conditions? 
Note that Lemma~\ref{lemma2} which is one of  the main tools in our investigation assumes boundedness of delays.
However, an analogue of  Lemma~\ref{lemma2} exists in the case of infinite but ``uniformly exponentially decaying" memory  \cite{AS}.

\item
All the stability conditions for equation (\ref{28}) are obtained 
using the reduction to the equation with two delays which in the proof of Theorem~\ref{theorem3} would 
correspond to the change of the variable
$$
s= \int_{t_0}^t \sum_{k=1}^m a_k(\zeta)~d\zeta.
$$
Is it possible to obtain different (and, in some cases, sharper) conditions, with one of the choices for the change 
of the variable 
$$
s= \int_{t_0}^t \sum_{k\in I} a_k(\zeta)~d\zeta, \quad \mbox{where} \quad I \subset \{1,2, \dots, k \} ?
$$

\item
In the present paper we derived explicit uniform exponential stability conditions
for delay differential equations, integro-differential equations and equations with
distributed delays. Obtain  explicit uniform exponential stability conditions for mixed type equations as corollaries of 
Theorems~\ref{theorem10}-\ref{theorem12},
for example, the following ones:
$$
\dot{x}(t)+a(t)x(h(t))-\int_{g(t)}^t  K(t,s)x(s)ds=0,
$$$$
\dot{x}(t)+\int_{h(t)}^t K(t,s)x(s)ds-b(t)x(g(t))=0,
$$$$
\dot{x}(t)+\sum_{k=1}^m a_k(t)x(h_k(t))-\sum_{k=1}^l\int_{g_k(t)}^t x(s)d_s T_k(t,s)=0.
$$
\item
Investigate global exponential stability for nonlinear equations considered in the paper.
Is it possible to claim (at least, under certain additional conditions)
that local stability implies existence of a global solution and
global stability (certainly, for positive initial conditions)?
\end{enumerate}

\subsection*{Acknowledgment}
The  second  author  was  partially  supported  by  the NSERC research grant RGPIN-2015-05976.
The authors are grateful to the anonymous referee whose valuable comments significantly contributed 
to the presentation of the results.


\begin{thebibliography}{99}
\bibitem{ABBD}
Agarwal,~R.~P.,~Berezansky,~L.,~Braverman,~E.~and Domoshnitsky,~A.,
{\it Non\-oscillation Theory of Functional Differential Equations with
  Applications}.  New York: Springer 2012.

\bibitem{AS}
Azbelev,~N.~V.~and  Simonov,~P.~M., {\it Stability of Differential
Equations with Aftereffect.}  Stability Control
Theory Methods Appl. 20.  London: Taylor $\&$ Francis
2003.

\bibitem{JMAA_lin_nonlin}
Berezansky,~L.~and~Braverman,~E.,
 On stability of some linear and nonlinear delay differential equations,
{\em J.~Math.~Anal.~Appl.} { 314} (2006), 391 -– 411.

\bibitem{BB2}
Berezansky,~L.~and~Braverman,~E.,
On exponential stability of  linear differential
equations with several delays.
{\it J.~Math.~Anal.~Appl.} { 324} (2006), \lb
1336 -– 1355.

\bibitem{BB3}
Berezansky,~L.~and~Braverman,~E.,
Explicit stability conditions for  linear differential equations with several delays.
{\it J.~Math.~Anal.~Appl.}  332 (2007), \lb
246 -- 264.

\bibitem{MCM2008}
Berezansky,~L.~and~Braverman,~E.,
Linearized oscillation theory for a nonlinear equation with a distributed delay.
{\em Math.~Comput.~Modelling}  48 
%(1), 
(2008), \lb
287 -- 304.

\bibitem{BB5}
Berezansky,~L.~and~Braverman,~E., 
Nonoscillation and exponential stability of delay differential equations with oscillating coefficients.
{\it J.~Dyn.~Control Syst.}  15 (2009), %no. 1, 
63 –- 82.

\bibitem{BB1}
Berezansky,~L.~and~Braverman,~E., 
New stability conditions for linear differential equations with several delays.
{\it Abstr. Appl. Anal.} (2011) Art.~ID 178568, 19 pp.


\bibitem{JMAA2017}
Berezansky,~L.~and~Braverman,~E.,
A note on stability of Mackey--Glass equations with two delays.
{\it J. Math. Anal. Appl.} 450 (2017),
1208 –- 1228.

\bibitem{BBI2013}
Berezansky,~L.,~Braverman,~E.~and Idels,~L., 
Mackey--Glass model of hema\-topoiesis with non-monotone feedback: stability, oscillation and control.
{\em Appl. Math. Comput.} 219 (2013), %no. 11, 
6268 -– 6283. 

\bibitem{BB4}
Berezansky,~L.~and~Braverman,~E.,
 On exponential stability of a linear delay differential equation with an oscillating coefficient.
{\em Appl. Math. Lett.} 22 (2009), %, no. 12, 
1833 –- 1837.

\bibitem{GD}
Gusarenko,~S.~A.~and Domoshnitsky,~A.~I.,
Asymptotic and oscillation properties of first-order linear scalar functional-differential
equations. {\it Diff.~Equ.} 25 (1989), 1480 -– 1491.

\bibitem{GH2}
Gy\"{o}ri,~I.~and Hartung,~F., Stability in delay perturbed differential and
difference equations. In: 
{\it  Topics in Functional Differential and Difference Equations} 
(Proceedings Lisbon 1999; eds.: T.~Faria et al.).
%Papers from the Conference on Functional Differential and Difference Equations held in Lisbon, July 26–30, 1999. Edited by  and Pedro Freitas. Fields Institute Communications, 29. American Mathematical Society, Providence, RI, 2001. xiv+378 pp. ISBN: 0-8218-2701-4 
{Fields Inst.~Commun.~29}. Providence (RI): Amer.~Math.~Soc.
 2001, pp.~181 -- 194.

\bibitem{GH1}
Gy\"{o}ri,~I.,~Hartung,~F.~and Turi,~J., Preservation of stability in delay
equations under delay perturbations. {\it J.~Math.~Anal.~Appl.}
220 (1998), 290 -- 312.

\bibitem{GL}
Gy\"{o}ri,~I.~and Ladas,~G.,
{\it Oscillation Theory of Delay Differential Equations}.
Oxford: Clarendon Press  1991.

\bibitem{Kz}
Krisztin,~T., On stability properties for one-dimensional
functional-differential equations.
{\it Funkcial. Ekvac.}  34 (1991), 241 -- 256.

\bibitem{Mak1}
Mackey,~M.~and Glass,~L., 
Oscillation and chaos in physiological control systems.
{\it Science} 197 (1977), 287 -- 289.

\bibitem{SY}
Shen,~J.~H.~and Yu,~J.~S.,
 Asymptotic behavior of solutions of neutral differential equations with positive and negative
coefficients. {\it J. Math. Anal. Appl.} 195 (1995), 517 -– 526.

\bibitem{SYC}
 So,~J.~W.\,H.,~Yu,~J.~S.~and Chen,~M.~P., Asymptotic stability for scalar
delay differential equations. {\it Funkcial. Ekvac.} 39 (1996),
1 -- 17.

\bibitem{W}
Wang,~T.,
Inequalities and stability for a linear scalar functional differential equation. {\it J. Math. Anal. Appl.} 298 (2004), 33 -– 44.

\bibitem{WL}
Wang,~X.~and Liao,~L., 
Asymptotic behavior of solutions of neutral differential equations with positive and negative
coefficients. {\it J. Math. Anal. Appl.}  279 (2003),  326 -– 338.

\bibitem{YS}
Yoneyama,~T.~and Sugie,~J.,
 On the stability region of scalar delay-differential equations. 
{\it J. Math. Anal. Appl.} 134 (1988),  408 –- 425.

\bibitem{ZW}
Zhang,~Z.~and Wang,~Z., 
Asymptotic behavior of solutions of neutral differential equations with positive and negative
coefficients. {\it Ann.~Diff.~Equ.}  17 (2001), 295 -– 305.

\bibitem{ZY}
Zhang,~Z.~and Yu,~J., 
Asymptotic behavior of solutions of neutral difference equations with positive and negative coefficients.
{\it Math. Sci. Res. Hot-Line} 2 (1998), 1 -– 12.

\end{thebibliography}
\end{document}